\documentclass[11pt]{amsart}
\usepackage{amsmath}
\usepackage{amssymb}
\usepackage{graphics}

\newcommand{\sig}{\sigma}

\newcommand{\lam}{\lambda}
\newcommand{\al}{\alpha}
\newcommand{\Om}{\Omega}
\newcommand{\R}{\mathbb{R}}
\newcommand{\C}{\mathbb{C}}
\newcommand{\Z}{\mathbb{Z}}
\newcommand{\D}{\mathbb{D}}

\numberwithin{equation}{section}

\theoremstyle{plain} 

\newtheorem*{schur}{Schur's Lemma}
\newtheorem{cor}[equation]{Corollary}

\theoremstyle{definition}

\theoremstyle{remark}
\newtheorem{rem}[equation]{Remark}

\begin{document}

\thanks{Partially supported by NSF Grant DMS-0074326}
\thanks{\today}

\title[Eigenfunctions on the Snowflake]
{Computing Eigenfunctions on the Koch Snowflake: \\ A New Grid and Symmetry.}

\author{John M. Neuberger}
\author{N\'andor Sieben}
\author{James W. Swift}

\email{
John.Neuberger@nau.edu,
Nandor.Sieben@nau.edu,
Jim.Swift@nau.edu}

\address{
Department of Mathematics and Statistics,
Northern Arizona University PO Box 5717,
Flagstaff, AZ 86011-5717, USA
}


\begin{abstract}
In this paper we numerically solve the eigenvalue problem
$\Delta u + \lambda u = 0$ on the fractal region defined by the Koch Snowflake,
with zero-Dirichlet or zero-Neumann boundary conditions.
The Laplacian with boundary conditions is
approximated by a large symmetric matrix.
The eigenvalues and eigenvectors of this matrix are computed by ARPACK.
We impose the boundary conditions in a way that gives improved accuracy over
the previous computations of Lapidus, Neuberger, Renka \& Griffith.
We extrapolate the results for grid spacing $h$ to the limit $h \rightarrow 0$ in order
to estimate eigenvalues of the Laplacian and compare our results to those of Lapdus et al.
We analyze the symmetry of the region to explain the multiplicity-two eigenvalues,
and present a canonical choice of the two eigenfunctions that span each
two-dimensional eigenspace.
\end{abstract}

\maketitle

\begin{section}{Introduction.}

In this paper we approximate solutions to the two eigenvalue problems

\begin{equation}
\label{lpde}
\begin{array}{rlcrl}
        \Delta u + \lambda u = 0 & \textrm{in } \Omega & \hspace{1in}
& \Delta u + \lambda u = 0 & \textrm{in } \Omega\\
        u = 0  & \textrm{on }  {\partial \Omega}  ~ \mbox{(D)}&
& \displaystyle{\frac{\partial u}{\partial \eta}} = 0  & \textrm{on }  {\partial \Omega} ~ \mbox{(N)},
\end{array}
\end{equation}
where $\Delta$ is the Laplacian operator, and
$\Omega\subset\R^2$ is the (open) region whose boundary $\partial\Omega$ is the Koch snowflake.
For convenience, we refer to $\Om$ as the {\em Koch snowflake region}.
The boundary conditions are zero-Dirichlet, or zero-Neumann, respectively.

These boundary value problems must be interpreted in the variational sense (see
\cite{lap91}) but we avoid the subtleties of functional analysis by discretizing the problem.
We use a triangular grid of points to approximate the snowflake region.
Then, we identify $u: \Om \rightarrow \R$ with
$u \in \R^N$, where $N$ is the number of grid points in $\Om$.
That is,
$$
u(x_i) \approx u_i
$$
at grid points $x_i \in \R^2$, $i \in \{ 1, 2, 3, \ldots , N\}$.
The {\em discretized Laplacian} is the symmetric matrix $L$, with the property
$$
(-\Delta u)(x_i) \approx (L u)_i = \sum_{j = 1}^N L_{i j} u_j .
$$
Of course, a specific grid and a scheme for enforcing the boundary conditions
are needed to define $L$.  This is described in Section 2.
Then, the eigenvalues and eigenfunctions of $L$ approximate the eigenvalues
and eigenfunctions defined by (\ref{lpde}).
The eigenvalues and eigenvectors of $L$ are our approximations of the eigenvalues
$0\leq \lambda_1<\lambda_2 \leq \lambda_3 \leq \cdots\leq\lambda_k\cdots\to\infty$
and the corresponding eigenfunctions
$\{\psi_k\}_{k=1}^\infty$ of the negative Laplacian $-\Delta$.

The Koch snowflake is a well known fractal, with Hausdorff
dimension $\log_3 4$. Following Lapidus, Neuberger, Renka, and
Griffith \cite{lnrg}, we take our snowflake to be inscribed in a
circle of radius $\frac{\sqrt{3}}{3}$ centered about the origin.
With this choice, the polygonal approximations used in the fractal
construction have side length that are powers of $1/3$. In
\cite{lnrg}, a triangular grid with spacing $h = h_{\rm LNR}(\ell)
= 1/3^\ell$ was used to approximate the eigenfunctions. Here
$\ell$ is a positive integer indicating the mesh size and the {\em
level} of polygonal approximation to the fractal boundary. With
this choice of $h$, there are $N_{\rm NLR}(\ell)  = 1 + (4 \cdot
9^\ell - 9 \cdot 4^\ell)/5$ grid points in $\Om$, as well as $3
\cdot 4^\ell$ grid points on $\partial\Om$ (see Table
\ref{grid_table}). The zero-Dirichlet boundary conditions are
imposed by setting $u_i = 0$ at the grid points on the boundary.

We used a different triangular grid. We
found {\em more accurate} results with a {\em larger} $h$ by choosing the
grid spacing to be $h = h_{\rm NSS}(\ell) = 2/3^\ell$ and placing the boundary
between grid points.
This yields $N = N_{\rm NSS}(\ell) = (9^\ell - 4^\ell)/5$ grid points in the
snowflake region $\Om$.
No grid points are on $\partial\Om$ with our choice.
We use {\em ghost points}, which are outside the region, to enforce the boundary conditions,
as described in Section \ref{ghost}.
To compare our results with those of \cite{lnrg} we will use
$\lam^{(\ell)}_k$ to denote the $k^{\rm th}$
eigenvalue of $L$ at level $\ell$ with our method, and
$\mu^{(\ell)}_k$ to denote the eigenvalues published in
\cite{lnrg}.

\begin{table}[htbp]
\label{grid_table}
\begin{center}
\begin{tabular}{|c|c|c|c|c|c|c|}
\hline
$\ell$ & 1 & 2 & 3 & 4 & 5 & 6 \\
\hline
$N_{\rm NSS}(\ell)$ & 1 & 13 & 133 & 1261 & 11605 & 105469 \\
\hline
$N_{\rm LNR}(\ell)$ & 1 & 37 & 469 & 4789 & 45397 & 417781
\\ \hline
\end{tabular}
\vspace*{.2in}
\caption{
The number of interior grid points with $h = h_{\rm NSS}(\ell) = 2/3^\ell$ (the current work),
and with $h = h_{\rm LNR}(\ell) = 1/3^\ell$ as in \cite{lnrg}.
Note that our new grid has approximately 75\% fewer grid
points at the same $\ell$, when $\ell$ is large.
The region $\Om$ is open and the larger values of $N$ published
in \cite{lnrg} includes the grid points on the boundary.
}
\end{center}
\end{table}

A different approach, avoiding triangular grids altogether, can
be found in the unpublished thesis \cite{banjai}.  This work
uses the conformal mappings found in \cite{bt03}.

In \cite{ns}, the Gradient Newton Galerkin Algorithm (GNGA) was developed to investigate
existence, multiplicity, nodal structure, bifurcation, and symmetry of problems
of the form (\ref{pde}).  (This PDE is found in the concluding section.)
The GNGA requires as input an orthonormal basis
of a sufficiently large
subspace consisting of eigenfunctions of the Laplacian.
Since our eventual application concerns solving the nonlinear equation (\ref{pde}) on
the region $\Omega$ with fractal boundary,
we face the considerable challenge of first obtaining eigenfunctions
(solutions to the linear problem (\ref{lpde})) numerically.
In \cite{lnrg}, this was done using essentially the inverse power method with deflation
on a triangular grid.
They used approximating boundary polygons with vertices at grid points.

Using our new grid, we improve upon their results and are substantially successful in
obtaining a basis of such functions for a sufficiently large subspace for our future nonlinear needs.
We use the sophisticated numerical package ARPACK instead of deflation.
This software is based upon an algorithmic variant of the Arnoldi process called the Implicitly
Restarted Arnoldi Method (see \cite{arpack}) and is
ideally suited for finding the eigen-pairs of the large sparse matrices associated with
the discretization $L$ of the Laplacian.
It is easily implemented, requiring only a user-provided subroutine giving the action
of the linear map.
One of our innovations in investigating the snowflake
is taking the boundary to lie between grid points and using
ghost points just outside $\Omega$ when approximating the Laplacian at interior points
closest to the boundary.
This results in better approximations of true eigenvalues using fewer interior grid points
than achieved by \cite{lnrg} using the standard grid method of enforcing the boundary condition.
We support this claim by comparing our results via curve fitting data points
to predict the true values.

In Section 2, we describe in more detail the triangular grid and the accompanying second difference scheme
for approximating the Laplacian,
as well as the ARPACK implementation using this information to generate the basis of eigenfunctions.
In Section 3, we compare our numerical eigenvalue approximations to those obtained in \cite{lnrg}.
In particular, we perform Richardson extrapolations on both data sets.
In Section 4, we
apply representation theory to determine
the 8 possible symmetries that eigenfunctions (and approximating eigenvectors)
can have, given the $\D_6$ symmetry of the region $\Om$ and the approximating grids.
We consider this rigorous treatment of symmetry to be a key contribution of this paper.
This information is used for numerical post-processing to find symmetric ``canonical'' representatives
for multiple eigenvalues.
This catalog of the symmetries of basis elements will be used in an essential way in our subsequent
nonlinear bifurcation studies.
Section 4 also contains graphics depicting a selection of approximating eigenvectors
to both problems in (\ref{lpde}).
Section 5 gives a brief indication of how our new grid can be used when implementing GNGA
on related nonlinear problems (see \cite{nss2}).
Also, we discuss how the known symmetries can be exploited to reduce the number of integrations required by that scheme.

\end{section}
\begin{section}{Ghost Points and ARPACK.}
\label{ghost}

In approximating the Laplacian for functions defined on $\Om$,
we developed the grid technique depicted in Figure \ref{grid}.
As in \cite{lnrg},
at interior points with interior point neighbors one sees that the standard second-differencing
scheme when applied to a triangular grid leads to the approximation
$$
-\Delta u (x) \approx \frac{2}{3h^2}\left(6u(x) - \sum\{\hbox{6 neighbor values of }u \}\right).
$$
Our scheme differs, however, when computing approximations at interior grid points with neighbors that lie
{\it outside} the boundary.
In \cite{lnrg} the value of zero at boundary points (which lie on their grid) is used
to enforce the zero-Dirichlet boundary condition.
When we approximate the Laplacian at a point $x_i$ near the boundary,
we set $u = -u(x_i)$ at ghost points which are neighbors of $x_i$.
Specifically, for the level $\ell=2$ example found in Figure \ref{grid} and with the understanding that
$u_i \approx u(x_i)$, we have
$$
\begin{array}{l}
\displaystyle{-\Delta u (x_1) \approx \frac{2}{3h^2}(6u_1-(u_2+u_3+u_4+u_5+u_6+u_7)) }, \\ \\
\displaystyle{-\Delta u (x_2) \approx \frac{2}{3h^2}(6u_2-((-u_2)+u_1+u_3+u_7+u_8+u_9))
= \frac{2}{3h^2}(7u_2-(u_1+u_3+u_7+u_8+u_9)), } \\ \\
\displaystyle{-\Delta u (x_9) \approx \frac{2}{3h^2}(6u_9-((-u_9)+(-u_9)+(-u_9)+(-u_9)+u_2+u_3))
= \frac{2}{3h^2}(10u_9-(u_2+u_3)). }
\end{array}
$$
In the first line, there are no ghost points used because all
the neighbors of $x_1$ are interior points.
In the approximation at $x_2$,
$(-u_2)$ represents the value of $u$ at the ghost point $g_1$, as labelled in Figure \ref{grid}.
In the last line, $(-u_9)$
represents the value of $u$ at $g_1,\ldots,g_4$.
Note that the value of $u$ at $g_1$ is different in the
calculation at $x_2$ and $x_9$.  An alternative way of imposing the boundary conditions is to set
$u(g_1)$ to be the average of $u_2$, $u_8$, and $u_9$.  We did experiments with this alternative method, but
the results were not as accurate as the method we have described.

Our method of imposing the zero-Dirichlet boundary conditions can be summarized as
$$
-\Delta u (x) \approx \frac{2}{3h^2}((12-(\hbox{number of interior neighbors}))u(x) -
\sum \{\hbox{interior neighbor values of } u\}).
$$
The zero-Neumann boundary condition can be easily applied as well.
Indeed, one has the only slightly different formula (which agrees away from the boundary)
$$
-\Delta u (x) \approx\frac{2}{3h^2}((\hbox{number of interior neighbors})u(x) -
\sum\{\hbox{interior neighbor values of } u\}).
$$
To enforce the Neumann instead of the Dirichlet condition,
we only need to change 3 characters of our ARPACK code.
Specifically, one deletes the ``$12-$'' found in the zero-Dirichlet formulae above.
\begin{figure}[ht]
\begin{center}
\scalebox{1.25}{\includegraphics{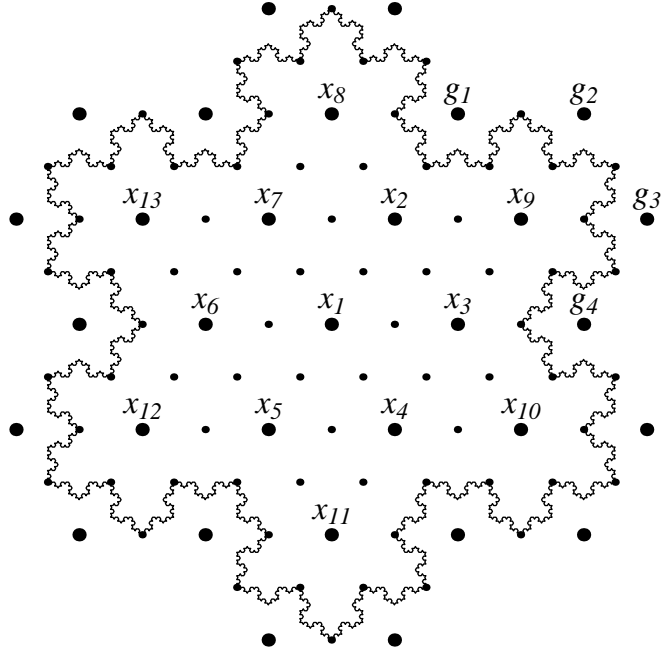}}
\caption{
The Koch snowflake $\partial\Om$ with
$N_{\rm NSS}(2)=13$ labelled grid points $\{x_i\}_{i=1}^{13}$ at level $\ell = 2$.
The grid used by
\cite{lnrg} consists of the $N_{\rm LNR}(2)=37$ large and small
points inside the snowflake, along with 48 small points on
$\partial \Om$.
The points outside of the snowflake, some labelled $g_i$, are
ghost points we use to enforce the boundary conditions.
For example,
$u(g_2) =-u(x_9)$ for Dirichlet boundary conditions and $u(g_2) = u(x_9)$
for Neumann boundary conditions.
On the other hand, $u$ takes on different values at $g_1$ when the Laplacian
is evaluated at $x_2$, $x_8$, or $x_9$.
}
\label{grid}
\end{center}
\end{figure}

The user-provided ARPACK subroutine takes as input a vector $v\in \R^{N}$ and outputs
$w\in \R^{N}$ with $w=Lv$ for the $N\times N$ matrix $L$ approximating the
discretized negative Laplacian, where for convenience we use $N$ to denote $N_{\rm NSS}(\ell)$.
This procedure is easily coded once an $N\times6$ dimensional array $t$ with neighbor
information is populated.
In the pseudocode in Figure~\ref{pseudocode},
$t(i,j)\in\{0,1,2,\ldots,N\}$ is the index of the $j^{\rm th}$ neighbor of the grid point $x_i$,
$j=1,\ldots,6$.
If $t(i,j)=0$ for some $j$, then the $i^{\rm th}$ interior point is near the boundary and has less than
6 interior point neighbors.
We let $k_i\in \{2,4,5,6\}$ denote the number of interior point neighbors of grid point $x_i$.
\begin{figure}[ht]
\begin{enumerate}
    \tt
    \item[] Loop for $i=1,\ldots,N$
    \begin{enumerate}
        \item[1.]  Set $$w(i) =\left\{\begin{array}{rl}
            (12-k_i)*v(i) & \ \hbox{for Dirichlet boundary conditions, or} \\
                   k_i*v(i) & \ \hbox{for Neumann}
                   \end{array}\right.$$
        \item[2.] Loop for $j=1,\ldots,6$
        \begin{enumerate}
            \item[a.] Find index $p = t(i,j)$ of $j^{\rm th}$ neighbor
            \item[b.] If $p\not=0$ then subtract neighbor value: $w(i) = w(i) - v(p)$
        \end{enumerate}
        \item[3.] Multiply by $h$ factor: $w(i) = 2 * w(i) / (3.0 * h*h)$
    \end{enumerate}
\end{enumerate}
\caption{Pseudo code for user-provided subroutine encoding the linear map
$v\mapsto w = Lv$.}
\label{pseudocode}
\end{figure}

The neighbor file is generated using the
\texttt{set} and \texttt{vector} data structures and the binary search
algorithm of the Standard Template Library in \texttt{C++}.
In the first step we find
the integer coordinates of the grid points
in the basis $\{(1,0),(\frac12,\frac{\sqrt 3}
{2})\}$.
The procedure uses simple loops to find the coordinates of
the grid points
inside a large triangle and
then calls itself recursively on three smaller triangles
placed on the three sides of the original triangle, until
the desired level is reached.  To avoid duplication of
grid points the coordinates are collected in a \texttt{set} data
structure.  In the second step, we copy the coordinates
into a \texttt{vector} data structure and use binary searches to
find the indices of the six possible neighbors of each
grid point.  In the last step, we compute the Cartesian
coordinates of the grid points and write them into a file
together with the indices of the neighbors.
\end{section}
\begin{section}{Numerical Results}

In this section we present our experimental results.
Our best approximations $\lambda_k^{\rm R}$ for the eigenvalues are obtained by performing
Richardson extrapolation. Specifically, we find the $y$-intercepts of the Lagrange polynomials fitting
the points $\{(h_{\rm NSS}(\ell),\lambda_k(\ell))\}_{\ell=4}^6$. We also compute Richardson extrapolations
$\mu_k^{\rm R}$ using the data published in \cite{lnrg}.
In Table~\ref{Richard_D1} we list the level 6 and Richardson approximations of the first ten
and the $100^{\rm th}$ eigenvalues.

\begin{table}[ht]
\begin{center}
\begin{tabular}{|c||c|c||c|c|}
\hline
 &NSS&NSS&LNR&LNR\\
$k$ & \vphantom{\rule[-.2cm]{0cm}{0cm}}$\lambda_k(6)$ & $\lambda_k^{\rm R}$ & $\mu_k(6)$ & $\mu_k^{\rm R}$\\
\hline\hline
1  &39.353  &39.349& 39.390&39.352  \\
2  &97.446  &97.438& 97.537&97.438  \\
3  &97.446  &97.438& 97.537&97.438  \\
4  &165.417 &165.409&165.622&165.478 \\
5  &165.417 &165.409&165.622&165.478\\
6  &190.381 &190.373&190.571&190.365\\
7  &208.622 &208.617&208.837&208.59 \\
8  &272.415 &272.413&272.755&272.480\\
9  &272.415 &272.413&272.755&272.480\\
10 &312.348 &312.358&312.645&312.351\\
100&2322.129&2324.925&&\\
\hline
\end{tabular}
\vspace*{.2in}
\caption{The first ten and the $100^{\rm th}$ eigenvalues to the Dirichlet problem.
NSS denotes our new grid, while LNR denotes the grid in \cite{lnrg}.
Provided are the level $\ell=6$ approximations,
where the NSS scheme uses $N_{\rm NSS}(6) = 105469$ grid points and
the LNR method uses $N_{\rm LNR}(6) = 417781$ interior grid points.
The symbols $\lambda_k^{\rm R}$ and $\mu_k^{\rm R}$ denote the
the Richardson extrapolation values for $\lambda_k$ using levels $\ell\in\{4,5,6\}$.
Blank entries correspond to no data available for comparison.}
\label{Richard_D1}
\end{center}
\end{table}

We computed the relative differences $\left(\lambda_k(6)-\lambda_k^{\rm R}\right)/\lambda_k^{\rm R}$
and $\left(\mu_k(6)-\mu_k^{\rm R}\right)/\mu_k^{\rm R}$ for $k\in\{1,\ldots,10\}$.
The relative differences of the $\lambda$ values ranged from $10^{-5}$ to $10^{-4}$, while
the relative differences of the $\mu$ values are all on the order of $10^{-3}$. The absolute
differences between the Richardson extrapolations $\lambda_k^{\rm  R}$ and $\mu_k^{\rm R}$
ranged from $10^{-4}$ to $7\cdot 10^{-2}$.
Note that even for $k=100$ we have
$\left(\lambda_{100}(6)-\lambda_{100}^{\rm R}\right)/\lambda_{100}^{\rm R}\approx-1.203\cdot 10^{-3}$.

In Figures~\ref{Rplot1} and \ref{Rplot2} we visually compare the Richardson extrapolations for $\lambda_1$
and $\lambda_{10}$. We can see that although the extrapolated values are nearly identical, our approximations
are much closer to the common extrapolated values using a lot fewer grid points. This is a key issue for
us since in \cite{nss2} we will require accurate eigenvectors and eigenvalues using as few grid points
as possible.
We cannot use the extrapolated eigenvalues,
since although they are more accurate approximations of eigenvalues of $-\Delta$,
they are not eigenvalues of $L$ corresponding to eigenvectors of $L$ at any given level.

\begin{figure}[ht] 
\begin{center}
\scalebox{.55}{\includegraphics{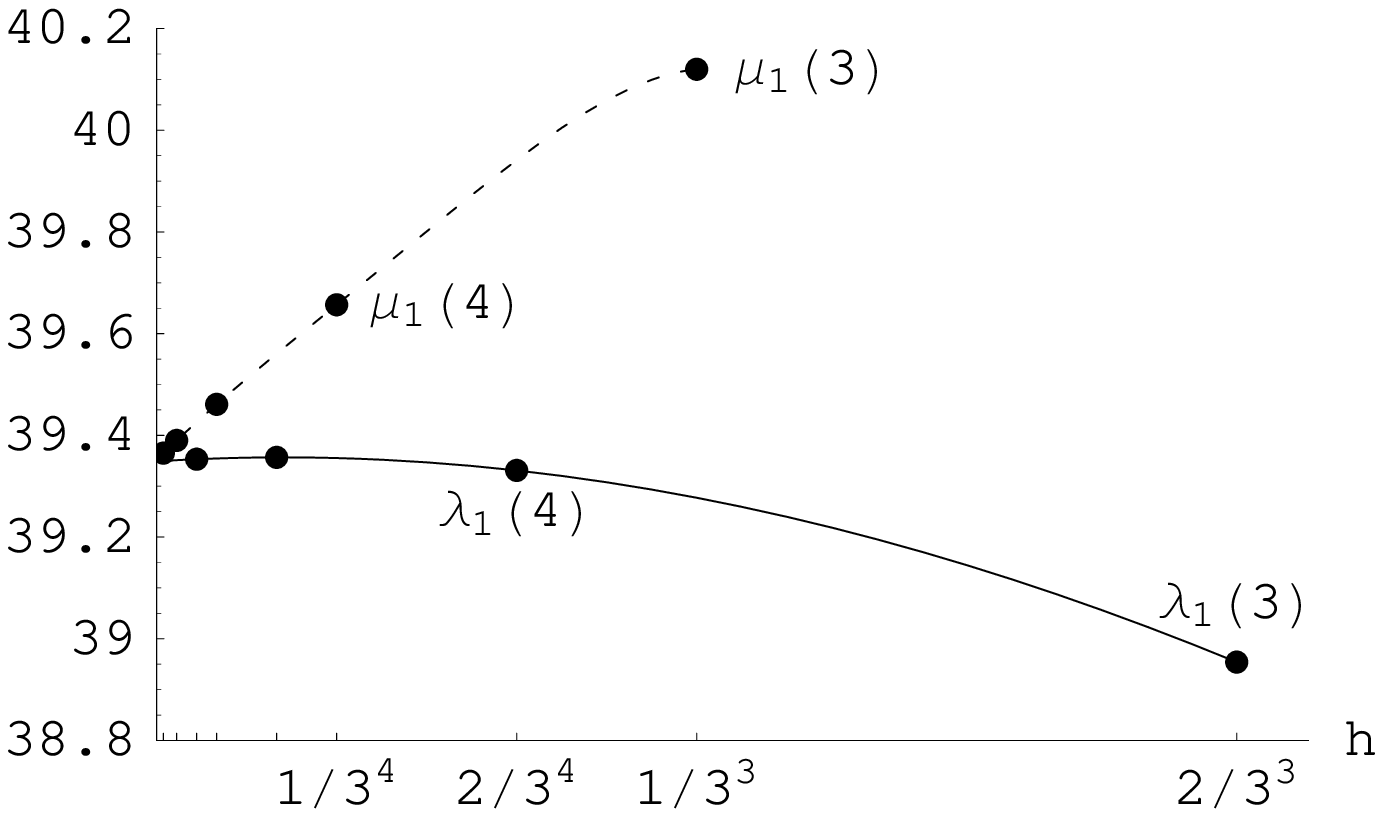}}
\hfill
\scalebox{.55}{\includegraphics{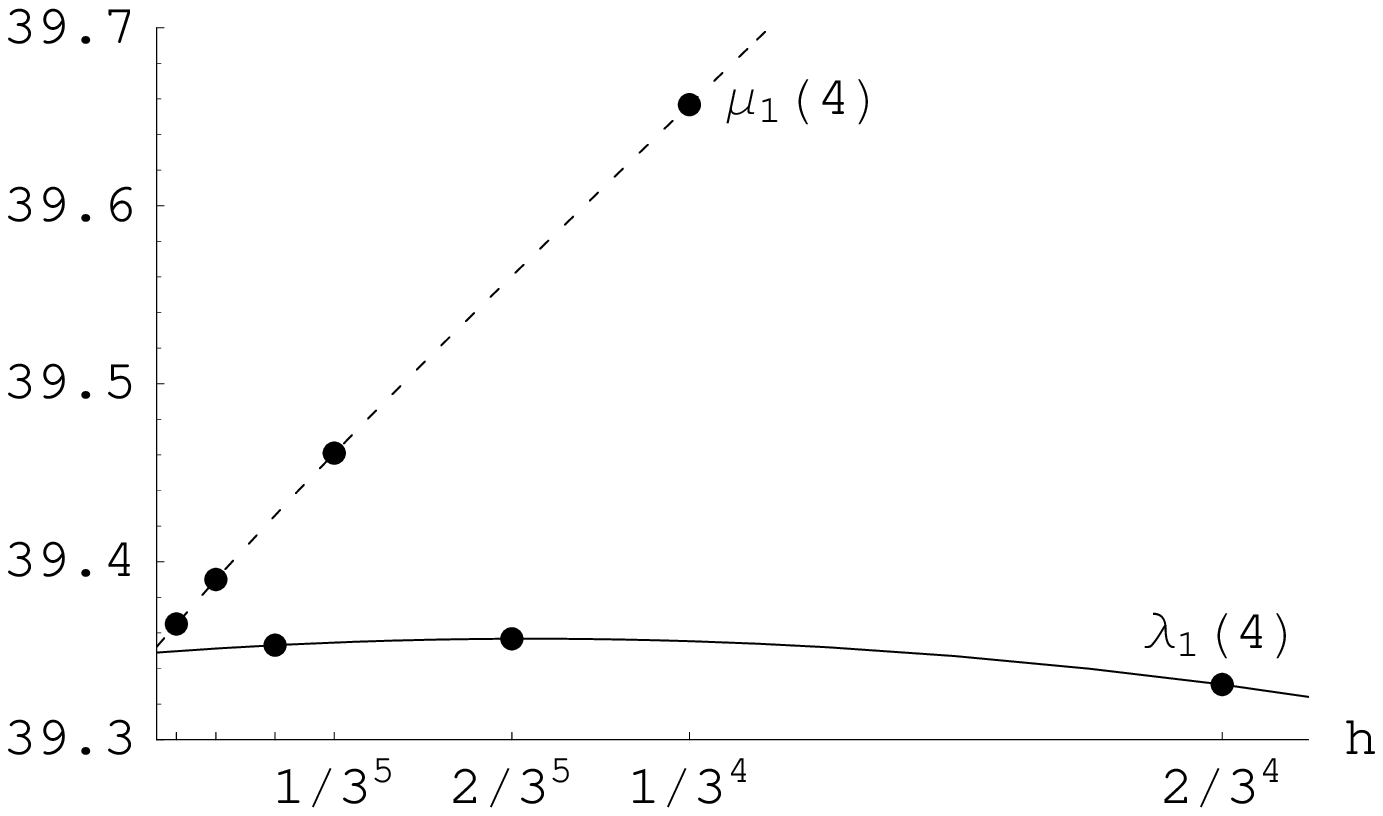}}
\caption{
Two views of the Richardson extrapolations for $\lambda_1$.
The solid line is the graph of the Lagrange polynomial fitting our data for $\ell\in\{3,4,5,6\}$.
The dashed line fits the $\ell\in\{3,4,5,6\}$ data of \cite{lnrg} together with the unpublished
level $\ell = 7$ eigenvalue approximation obtained via
private communication from Robert Renka.
}
\label{Rplot1}
\end{center}
\end{figure}

\begin{figure}[ht]  
\begin{center}\
\scalebox{.55}{\includegraphics{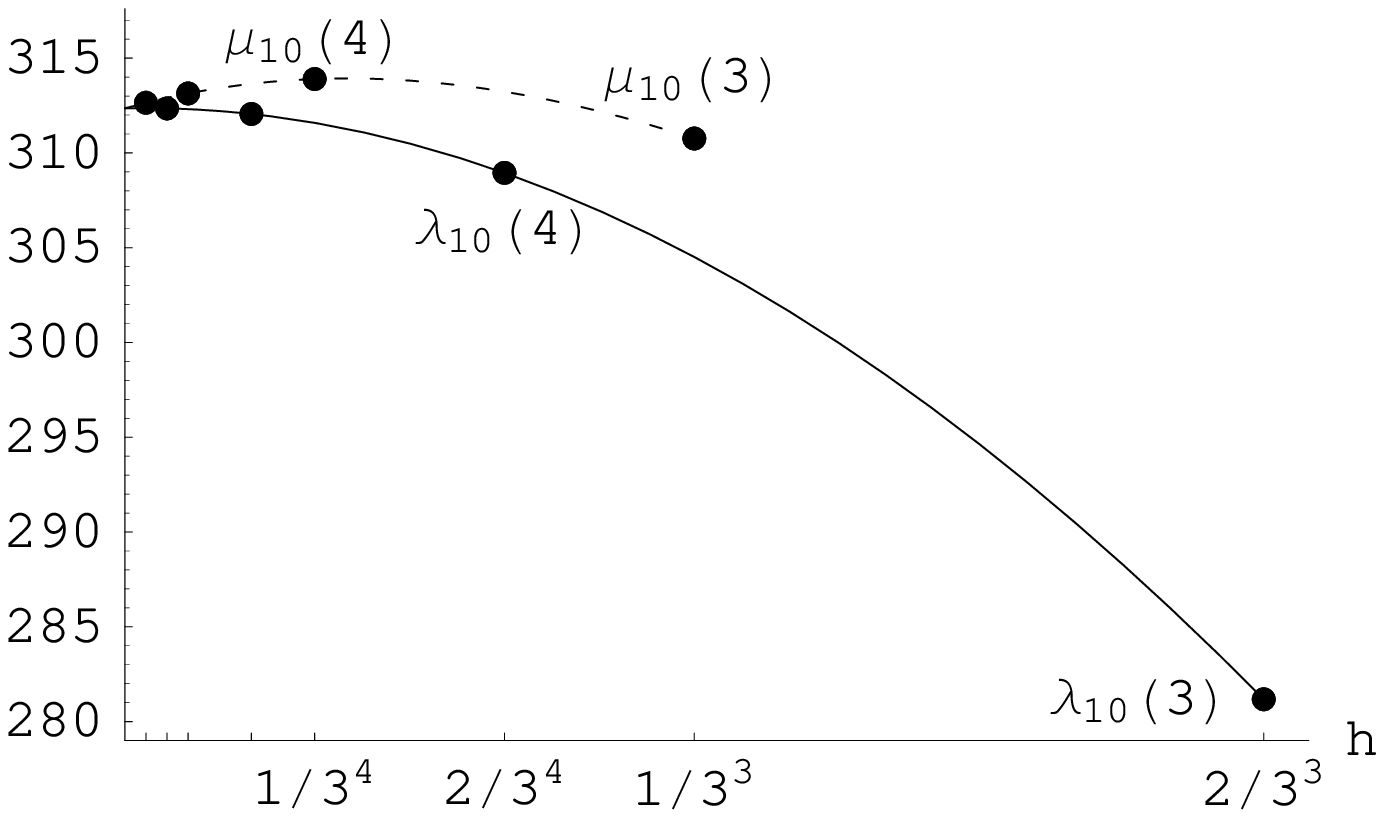}}
\hfill
\scalebox{.55}{\includegraphics{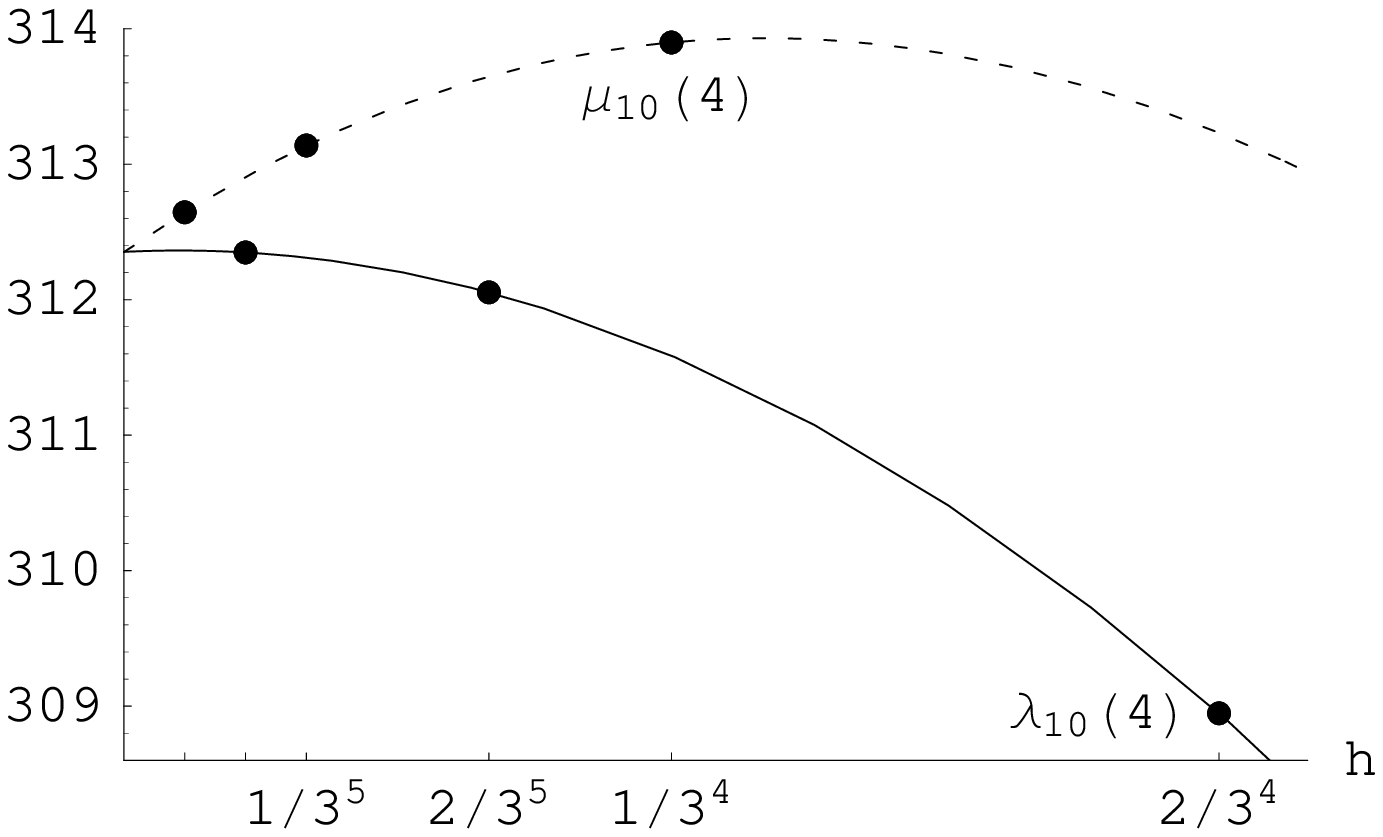}}
\caption{
Two views of the Richardson Extrapolations for $\lambda_{10}$.
As in Figure \ref{Rplot1}, the solid lines correspond to our data and
the dashed lines to the data in \cite{lnrg}.
}
\label{Rplot2}
\end{center}
\end{figure}

Our results for the Neumann boundary conditions are shown in Table~\ref{Richard_N}.
Based on private communication we know that our approximate eigenvalues are very close
to the unpublished numbers obtained by Lapidus et al. A careful comparison of the two
grid schemes is not possible at this time, since we do not have all of their data.
We found that the Lagrange polynomial of our Neumann data approaches $h = 0$ linearly,
similar to the curves of \cite{lnrg} in Figures \ref{Rplot1} and \ref{Rplot2}.
This led us to consider an alternate
scheme for enforcing boundary conditions. As noted in Figure~\ref{grid}, $u$ is multi-valued
at certain ghost points. We tried using a single average value at these ghost points, but
the slope of the Lagrange polynomial at 0 was larger, so that the eigenvalues at a given level were farther from
the extrapolated value.
It is an area for future research to understand why the ghost points
work so well.
All we can now assert is that our method clearly out-performs that found in \cite{lnrg} for the zero-Dirichlet
eigenvalue problem on the snowflake region.
The ghost points can be used in general regions, and it would be interesting
to determine the optimal method for enforcing
boundary conditions on general regions.

\begin{table}[ht]
\label{Richard_N}
\begin{center}
\begin{tabular}{|c|c|c|c|}
\hline
\vphantom{\rule[-.2cm]{0cm}{.6cm}}
$k$&$\lambda_k(6)$ & $\lambda_k^{\rm R}$  &
$(\lambda_k(6) - \lambda_k^{\rm R})/\lambda_k^{\rm R}$ \\
\hline
  1 &    0.0000 &   0.0000    & NA      \\
  2 &   11.9105 &  11.8424    & 0.0057  \\
  3 &   11.9105 &  11.8424    & 0.0057  \\
  4 &   23.1770 &  23.0466    & 0.0057  \\
  5 &   23.1770 &  23.0466    & 0.0057  \\
  6 &   27.5770 &  27.4261    & 0.0055  \\
  7 &   52.4164 &  52.2105    & 0.0039  \\
  8 &   85.8449 &  85.5521    & 0.0034  \\
  9 &   85.8449 &  85.5521    & 0.0034  \\
 10 &  112.7801 &  112.0200   & 0.0068  \\
100 & 1295.4431 & 1271.1900   & 0.0191  \\
\hline
\end{tabular}
\vspace*{.2in}
\caption{
The first 10 and $100^{\rm th}$ eigenvalues for the Neumann problem.
We have included the level $\ell=6$ approximations, the Richardson extrapolations
using $\ell\in\{4,5,6\}$, and the relative differences of the two.
All results use our new grid. The Neumann eigenvalues obtained by
Lapidus et al. have not been published.
}
\end{center}
\end{table}

We produced contour plots of the eigenfunctions.  An example is shown in
Figure~\ref{sixth}.  These contour plots were produced by a Mathematica notebook
that reads in the $u$ vector and outputs a postscript file.  The level of the
grid approximation is computed from the length of the $u$ vector.  All of the
contour plots shown in this paper use $\ell = 5$ data for which the $u$ vector has
length $N_{\rm NSS}(5) = 11605$.
\begin{figure}[ht]   
\begin{center}\
\scalebox{.7}{\includegraphics{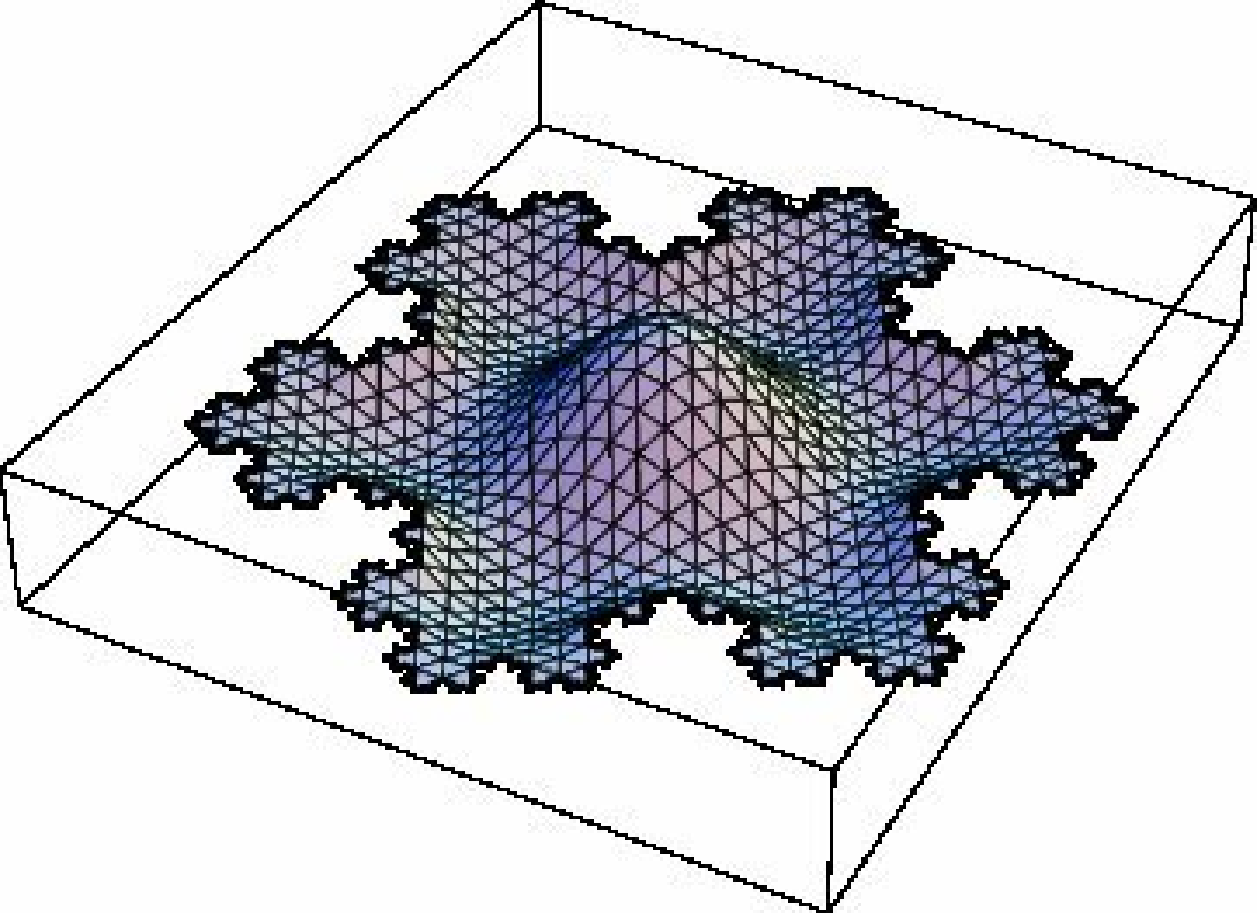}}
\hfill
\scalebox{.4}{\includegraphics{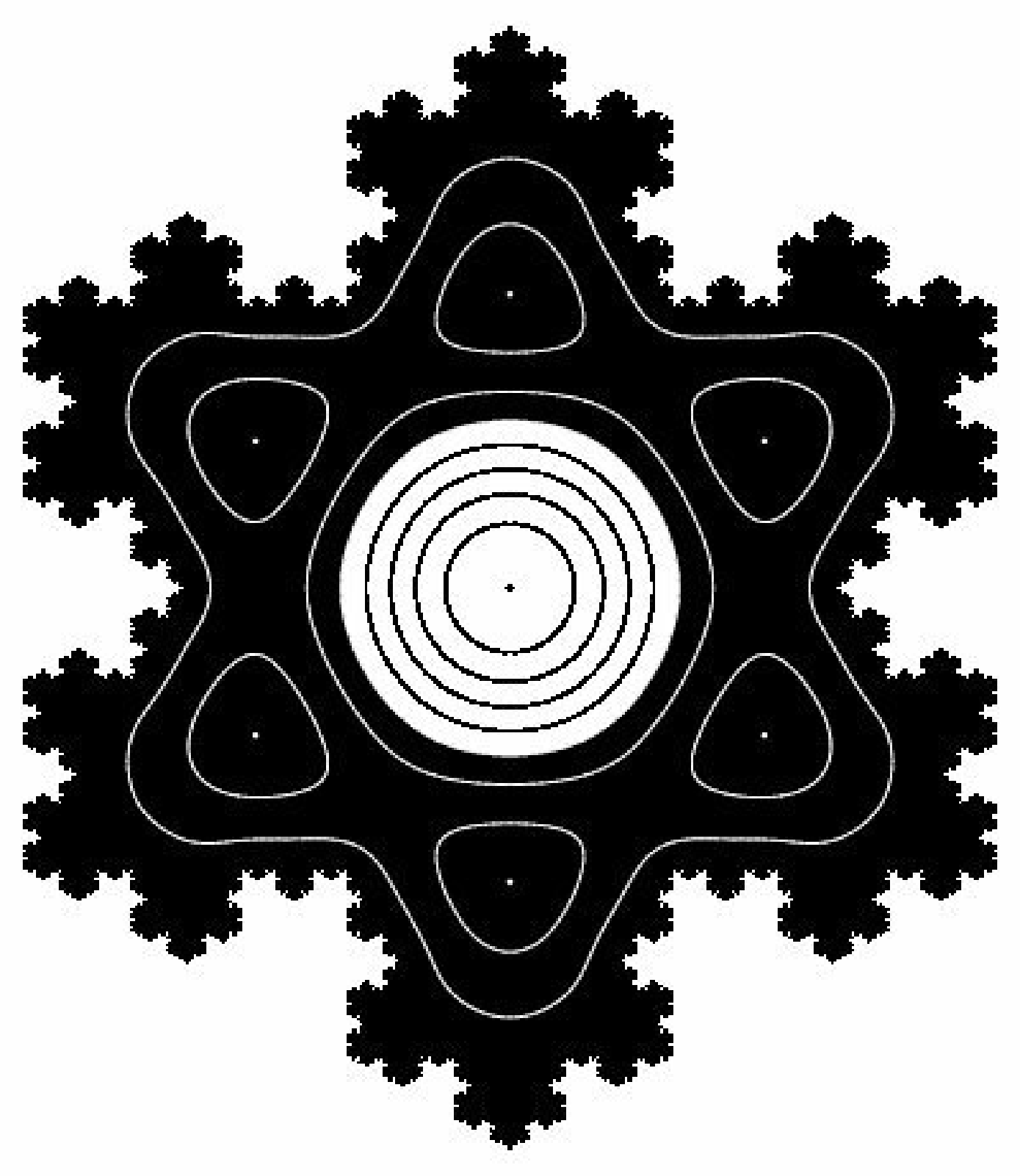}}
\caption{ The graph, and a contour plot, of the sixth
eigenfunction, $\psi_6$, with Dirichlet boundary conditions.
The graph uses 469 grid points to triangulate the snowflake region.
The contour plot shows our level $\ell = 5$ data with the new grid,
as described in the text.
The white and black regions in the contour plot represent positive and
negative values of $\psi_6$.  The contours are equally spaced, and the
dots represent local extrema of $\psi_6$.
\label{sixth}
}
\end{center}
\end{figure}

The local extrema of $u$ are calculated in two steps.  First the extreme values of $u_i$
are calculated.  Then, a quadratic fit to this data point and its six neighbors is performed.
A dot is then drawn at the extremum of the quadratic function.  This extra effort, compared to
drawing a dot at the grid point, has a noticeable effect even at level 5.
After the extrema are found, the $u$ values of the contours are computed using a heuristic
that gives fewer levels as the number of extrema increases.

The black regions are then drawn by subdividing the snowflake region into the triangles
defined by the grid points.  If $u < 0$ at all three vertices of a triangle then the triangle
is filled with black.  If $u < 0$ at some vertices of the triangle and $u > 0$ at others,
then a linear interpolation is used to estimate the region where $u<0$ within the triangle.
The contours are also produced using linear interpolation within the triangles:  If the value
of $u$ on the vertices spans a contour value, then a short line segment inside the triangle
is drawn based on the linear fit.

Several details of the implementation of the contour plotting have been left out.  For example
the region outside of the triangulation of grid points, but inside the snowflake boundary,
is shaded by a different technique.

\end{section}
\begin{section}{Symmetry and the Canonical Basis.}

Some of the eigenvalues of the Laplacian on the snowflake, (\ref{lpde}), have multiplicity one,
and some have multiplicity two.  In this section we quote well-known results
in group representation theory to explain the observed multiplicity.
We also describe a \emph{canonical} way to
choose two eigenvectors to span the two-dimensional eigenspaces.
Details of group representation theory can be found in \cite{tinkam}, \cite{sternberg}
and \cite{scott}.

Assume that $G$ is a finite group. A {\it linear representation\/} of $G$ is a
homomorphism $\alpha:G\to GL(U)$
where $GL(U)$ is the group of invertible
linear operators on the vector space $U = \R^N$ or $\C^N$.
The vector space $U$ is called the {\it representation space\/} of the linear
representation.
If $B$ is a basis for $U$ and $T\in GL(U)$ then we write
$[T]_B$ for the matrix of $T$ in the basis $B$ or simply $[T]$
if $B$ is the standard basis. We call the map $g\mapsto [\alpha (g)]$
a \emph{matrix representation}.
Two representations $\alpha ,\beta :G\to H$ are
{\it equivalent\/}, and we write $\alpha \sim \beta$, if there is an $h\in H$ such that
$\beta (g)=h^{-1}\alpha (g)h$ for all $g\in G$.

Let $\alpha :G\to GL(U)$ be a linear representation.  If $g\in G$ then
$\alpha (g): U \to U$ is a linear operator; for convenience we sometimes use
the notation $\alpha_g=\alpha (g)$.
The linear representation $\al$ induces a {\em group action}
$G \times U \to U$.  We often write $g \cdot u$ in place
of $\al_g (u)$ when the representation $\al$ is understood.

A subspace $W$ of $U$ is called an
{\em invariant subspace} of $\alpha$ if $\alpha_g(W)\subseteq W$ for all $
g\in G$.  The
representation $\alpha$ is called {\em irreducible} if $\alpha$ has no proper
invariant subspaces.
The property of complete reducibility, also known as
Maschke's theorem, says that there are $\alpha$-invariant
subspaces $U_1,\ldots , U_k$ such that $U=U_1\oplus\cdots\oplus U_{k}$ and
$\gamma^{(n)} := \alpha |_{U_n}$ is irreducible for each $n \in \{1, \ldots, k \}$.
If $B_n$ is a basis for $U_n$ and
$B=\cup_n B_n$ then the matrix of $\alpha_g$ in the basis $B$ is block
diagonal for all $g\in G$, that is
$$[\alpha_g]_B=\bigoplus_{n=1}^k[\gamma^{(n)}(g)]_{B_n}.$$

Let $\Gamma^{(i)}$, $i \in\{ 1, \ldots , q\}$
be an element from each of the $q$ equivalence classes of
irreducible representations of $G$.  Suppose we have a complete decomposition
of the representation $\alpha$ into irreducible representations $\gamma^{(n)}$.
For each $i$ there is an $\alpha$-invariant subspace
$$
V^{(i)} = \bigoplus \{ U_n \mid \gamma^{(n)} \sim \Gamma^{(i)} \} .
$$
Whereas there is great freedom in choosing the elements $U_n$ in
$U=U_1\oplus\cdots\oplus U_{k}$,
the decomposition $U=V^{(1)}\oplus\cdots\oplus V^{(q)}$ is unique up to ordering.

The {\em characters} of the representation $\Gamma^{(i)}$ are
$\chi^{(i)}(g)= \mbox{tr} [ \Gamma^{(i)} (g) ]$.
These characters are used in projection operators
\begin{equation}
\label{projection}
P^{(i)} ={\frac{d_i}{|G|}}\sum_{g\in G} \chi^{(i)}(g)\alpha_g
\end{equation}
onto the invariant subspaces $V^{(i)}= P^{(i)} (U)$
where $d_i $ is the dimension of the $i^{\rm th}$ irreducible representation.
Note that $[\Gamma^{(i)}(g)]$ is a $d_i \times d_i$ matrix.

To proceed we must choose a fixed set of matrix representations $[\Gamma^{(i)}]$ from each equivalence
class of irreducible representations.
We call these {\em canonical} matrices.  The following calculations
are simplified if we make the canonical
matrices as simple as possible.
It is always possible to choose matrices that are
unitary, and
we assume that the canonical matrices are unitary.

There is set of projection
operators described in \cite{tinkam}
\begin{equation}
\label{projection2}
P^{(i)}_j ={\frac{d_i}{|G|}}\sum_{g\in G} [\Gamma^{(i)}(g)]_{j,j} \alpha_g ,
\end{equation}
where the coefficient of $\al_g$ is the $j^{\rm th}$ diagonal element
of the matrix $[\Gamma^{(i)}(g)]$.
Note that $P^{(i)} = \sum_{j=1}^{d_i} P^{(i)}_j$.
The corresponding vector spaces
$
V^{(i)}_j = P^{(i)}_j (U)
$
are not $\al$-invariant, but they are orthogonal, and we shall see that
the following decomposition of the representation space is very useful:
\begin{equation}
\label{fullDecomposition}
U = \bigoplus_{i=1}^q \left ( \bigoplus_{j=1}^{d_i} V^{(i)}_j \right ) .
\end{equation}

We now consider the effect of the symmetry on commuting linear operators.  The
theory is much simpler if we assume that the representation space $U$ is complex.
\begin{schur}
Suppose
$\Gamma_1$ and $\Gamma_2$ are two irreducible representations of $G$ on $\C^{d_1}$ and
$\C^{d_2}$, respectively, and
$S: \C^{d_1} \to \C^{d_2}$ is a linear operator
such that $S \Gamma_1(g) = \Gamma_2(g) S$
for all $g \in G$.
Then $S = 0$ if $\Gamma_1$ and $\Gamma_2$ are not equivalent, and
$S = c I$ for some $c \in \C$  if $\Gamma_1=\Gamma_2$.
\end{schur}

Note that Schur's Lemma does not address the case where $\Gamma_1$ and $\Gamma_2$ are
different, but equivalent, representations.  This case can be addressed by choosing a
basis where $\Gamma_1$ and $\Gamma_2$ are the same, as in the proof of the following
corollary.
An abstract version of this corollary can be found in \cite[Theorem~1.2.8]{sagan}.

\begin{cor}
\label{SchurCor}
Suppose that $U = \C^{N}$ is a representation space for $\al$ that is decomposed as in
(\ref{fullDecomposition}),
and $T: U \to U$ is a linear operator
that commutes with $\al$.  That is, $\al_g T = T \al_g$ for all $g \in G$.
Then each of the spaces $V^{(i)}_j$ is $T$-invariant, and
the operators $T^{(i)}_j := T|_{V^{(i)}_j}$ decompose $T$ as
\begin{equation}
\label{TDecomposition}
T = \bigoplus_{i=1}^q \left ( \bigoplus_{j=1}^{d_i} T^{(i)}_j \right ) .
\end{equation}
Furthermore, $T^{(i)}_{j}$ is similar to $T^{(i)}_{j'}$.
\end{cor}

\begin{proof}
By complete reducibility, $U = \bigoplus_{n=1}^k U_n$, where $\al|_{U_n}$ is equivalent
to the irreducible representation $\Gamma^{(i_n)}$, with
$i_n \in \{1, \ldots, q \}$.  We can write $T$ in terms of $k^2$ blocks
$T_{m,n}=P_{U_m}\circ T|_{U_n}: U_n \to U_m$.
%
For each $n$,
we can choose a basis $U_n = \mbox{span} \{ e_{n,j} \mid j = 1, \ldots , d_{i_n} \}$
such that
$P^{(i)}_j e_{n,j'} = \delta_{i, i_n} \delta_{j, j'} e_{n,j}$,
and $\al_g(e_{n,j}) = \sum_{j'=1}^{d_{i_n}} e_{n,j'} [\Gamma^{(i_n)}(g)]_{j',j}$.
In other words, the basis vectors $e_{n,j} \in U_n$ transform like the standard basis vectors
$e_j \in \R^{d_{i_n}}$ do when multiplied by the matrices $[\Gamma^{(i_n)}(g)]$.
If $i_n = i_m$, then $\alpha |_{U_n} = \alpha |_{U_m}$.
Thus, Schur's Lemma implies
$$
T(e_{n,j}) = \sum_{m=1}^k  e_{m, j} \, c_{m, n}
$$
where $c_{m, n} = 0$ if $i_m \neq i_n$.
So, if $i_m = i_n$ then $T(e_{n,j})$ is a sum over $e_{m,j}$ with the same $j$.
The spaces $V^{(i)}_j$ that we defined in terms of projection operators can be written as
$$
V^{(i)}_j = \mbox{span} \{ e_{n,j} \mid i_n = i \},
$$
and it is now clear that $V^{(i)}_j$ is $T$-invariant.
Finally, $T^{(i)}_{j}$ is similar to $T^{(i')}_{j'}$ if $i = i'$,
since in the basis we have constructed they are represented by the same matrix.
\end{proof}

\begin{rem}
The spectrum of $T$ is the set of eigenvalues of the $k \times k$ matrix $C$, with elements
$c_{m,n}$.
The matrix $C$
separates into $q$ diagonal blocks $C^{(i)}$, one for each irreducible representation class of $G$.
If $\lam$ is an eigenvalue of $C^{(i)}$, then $\lam$ is an eigenvalue of $T$
with multiplicity at least $d_i$.
If the matrix $C$ has any multiple eigenvalues, then we say that $T$ has {\em accidental
degeneracy}.  In this case, a perturbation of $T$ that commutes with $\al$ can be chosen
so that the perturbed $C$ matrix has simple eigenvalues.
Barring accidental degeneracy, every eigenvalue of $T$ corresponds to a unique
irreducible representation $\Gamma^{(i)}$, the eigenvalue has multiplicity $d_i$,
and an orthonormal set of eigenvectors of this eigenvalue can be chosen,
one from each of the $d_i$ subspaces
$V^{(i)}_j$, $j \in \{ 1, 2, \ldots , d_i \}$.
These $d_i$ eigenvectors can be used as a basis for one irreducible space $U_n$.
In this way, we can write
$U = \bigoplus_{n=1}^k U_n$, where each irreducible component $U_n$ is
the eigenspace of an eigenvalue $\lam_n$ of $T$.
\end{rem}

If we require the spectrum
of a linear operator $T: \R^N \to \R^N$ that commutes with a representation on a {\em real} vector space,
we first complexify to the operator $\tilde{T}:\C^N \to \C^N$.
The spectra of $T$ and $\tilde{T}$ are the same, so the consequences of Schur's Lemma
hold {\em provided} we consider representations
$\Gamma^{(i)}$ which are irreducible over the field $\C$.
Sometimes an irreducible representation over $\R$ breaks up into two complex conjugate
irreducible representations over $\C$.  The eigenvalues of $T$
associated with this pair of representations
are complex conjugates.  If $T$ is self-adjoint, and thus has real eigenvalues,
this causes ``accidental degeneracies.''
For the Laplacian operator on the snowflake domain we do not have to consider this
complication since the irreducible representations
we need are the
same over the field of real or complex numbers.

We now apply the general theory to the eigenvalue equation (\ref{lpde}).
The symmetry group of the snowflake region, as well as the set of grid points in $\Om$,
is the dihedral group
$$
\D_6 = \langle \rho, \sigma \mid \rho^6 = \sigma^2 = 1, \rho \sigma = \sigma \rho^5 \rangle .
$$
It is convenient to define $\tau = \rho^3 \sigma$. Note that $\sig \tau = \tau \sig = \rho^3$, and $\rho^3$ commutes
with every element in $\D_6$.
The standard action of $\D_6$ on the plane is
\begin{align}
\begin{split}
\label{actionOnPlane}
\rho \cdot(x,y)  &= \mbox{$
   \left ( \frac{1}{2} x + \frac{\sqrt{3}}{2} y, -\frac{\sqrt{3}}{2} x + \frac{1}{2} y \right )$} \\
\sig \cdot(x,y)  &= (-x, y) \\
\tau \cdot(x,y)  &= (x, -y) .
\end{split}
\end{align}
In this action $\rho$ is a rotation by $60^\circ$, $\sig$ is a reflection across the $y$-axis,
and $\tau$ is a reflection across the $x$-axis.

For a given grid with $N$ points in $\Omega$, the $\D_6$ action on the plane (\ref{actionOnPlane})
induces a group action on the integers $\{1, 2, 3, \ldots , N\}$ defined by $x_{g\cdot i}=g\cdot x_i$.
There is also a natural action on the space of all functions from $\Om$ to $\R$, $U = \R^N$, defined by
$(g \cdot u)(x_i) = u( g^{-1} \cdot x_i)$
for all $u \in U$ and $g \in \D_6$.
With the usual identification $u_i = u(x_i)$,
this action corresponds to a linear representation $\alpha$ of  $\D_6$ on $U = \R^N$ defined by
$ (\alpha_g(u))_i=u_{g^{-1}\cdot i}$.


There are exactly 6 irreducible representations
of $\D_6$ up to equivalence.  Our canonical matrices are listed in Table \ref{irreps}.
\begin{table}
\begin{center}
\begin{tabular}{|c|ccc|}
\hline
$i$ & $[\Gamma^{(i)}(\rho)]$ & $[\Gamma^{(i)}(\sig)]$ & $[\Gamma^{(i)}(\tau)]$\\
\hline
$1$ & 1   &  1  & 1 \\ \hline
$2$ & 1   &  $-1$  & $-1$ \\ \hline
$3$ & $-1$   &  1  & $-1$ \\ \hline
$4$ & $-1$   &  $-1$  & 1 \\ \hline
$5$ & $\left(
\begin{array}{cc}
-1/2 & \sqrt{3}/2 \\
-\sqrt{3}/2 & -1/2
\end{array}
\right)$   &  $\left(
\begin{array}{cc}
1& 0 \\
0 & -1
\end{array}
\right)$  & $\left(
\begin{array}{cc}
1& 0 \\
0 & -1
\end{array}
\right)$ \\ \hline
$6$ & $\left(
\begin{array}{cc}
1/2 & \sqrt{3}/2 \\
-\sqrt{3}/2 & 1/2
\end{array}
\right)$   &  $\left(
\begin{array}{cc}
1& 0 \\
0 & -1
\end{array}
\right)$  & $\left(
\begin{array}{cc}
-1& 0 \\
0 & 1
\end{array}
\right)$ \\ \hline
\end{tabular}
\vspace*{.2in}
\caption{
\label{irreps}
A representative from each of the six equivalence classes of
irreducible matrix representations of $\D_6$.
These canonical
matrices are real, but the representations are nevertheless irreducible over the complex numbers.
The homomorphism condition $\Gamma^{(i)}(g h) = \Gamma^{(i)}(g) \Gamma^{(i)}(h)$ allows
all of the matrices to be computed from $[\Gamma^{(i)}(\rho)]$ and $[\Gamma^{(i)}(\rho)]$.
The last column is included for convenience.
The representation $\Gamma^{(6)}$ corresponds to the standard action of $\D_6$ on the plane (\ref{actionOnPlane}).
}
\end{center}
\end{table}
Since we have chosen a set of real matrices, $P^{(i)}$ and $P^{(i)}_j$ are projection operators
on both $\R^N$ and $\C^N$.  We can write the representation space $U = \R^N$ as
$$
U = V^{(1)} \oplus V^{(2)} \oplus V^{(3)} \oplus V^{(4)} \oplus V^{(5)}_1 \oplus V^{(5)}_2
\oplus V^{(6)}_1 \oplus V^{(6)}_2,
$$
where each of the real vector spaces $V^{(i)}$ and $V^{(i)}_j$ is invariant under any commuting linear operator
on $U$.
There are just six spaces in this decomposition regardless of the number of grid points $N$.
On the other hand, the decomposition $U = \sum_{n=1}^k U_n$ has thousands of irreducible components $U_n$
when $N$ is large.

In the remainder of this section
we use $T$ to refer to the negative Laplacian {\em operator}.
The results of Corollary \ref{SchurCor} hold, since $T$ commutes with the $\D_6$ action on $U$.
We will use the terms eigenvectors and eigenfunctions of $T$ interchangeably, since
the vector $u \in \R^N$ is a function on the $N$ grid points.

The numbering, 1 through 6, of the irreducible representations of $\D_6$ is somewhat arbitrary.
Therefore we give the names $V_{p_x p_y d}$ to the 8 $T$-invariant spaces $V^{(i)}_j$,
where $p_x$ and $p_y$ describe the parity of the functions under
the reflections $\sig$ and $\tau$ respectively, and $d$ is the dimension of the associated irreducible representation.
Some calculations allow us to give simple descriptions these $T$-invariant spaces
without having to appeal to the projection operators:
\begin{align*}
V_{++1} := V^{(1)} &= \{u \in U \mid \rho \cdot u = u, ~ \sig \cdot u = u , ~ \tau \cdot u = u \} \\
V_{--1} := V^{(2)} &= \{u \in U \mid \rho \cdot u = u, ~ \sig \cdot u = -u , ~ \tau \cdot u = -u \} \\
V_{+-1} := V^{(3)} &= \{u \in U \mid \rho \cdot u = -u, ~ \sig \cdot u = u , ~ \tau \cdot u = -u \} \\
\displaybreak[0]
V_{-+1} := V^{(4)} &= \{u \in U \mid \rho \cdot u = -u, ~ \sig \cdot u = -u , ~ \tau \cdot u = u \} \\
V^{(5)} &= \{u \in U \mid \rho^3 \cdot u = u, ~ u + \rho^2 \cdot u + \rho^4 \cdot u = 0 \} \\
\displaybreak[0]
V^{(6)} &= \{u \in U \mid \rho^3 \cdot u = -u, ~ u + \rho^2 \cdot u + \rho^4 \cdot u = 0 \} \\
V_{++2} := V^{(5)}_1 &= \{u \in V^{(5)} \mid \sig \cdot u = u , ~ \tau \cdot u = u \} \\
V_{--2} := V^{(5)}_2 &= \{u \in V^{(5)} \mid \sig \cdot u = -u , ~ \tau \cdot u = -u \} \\
V_{+-2} := V^{(6)}_1 &= \{u \in V^{(6)} \mid \sig \cdot u = u , ~ \tau \cdot u = -u \} \\
V_{-+2} := V^{(6)}_2 &= \{u \in V^{(6)} \mid \sig \cdot u = -u , ~ \tau \cdot u = u \}
\end{align*}

We can use Corollary \ref{SchurCor} in two different ways.  First, we could construct the
numerical approximations to the restricted operators $T^{(i)}_j$ and find the eigenvalues and eigenvectors
of the corresponding matrices.  This has the advantage that the matrices which represent $T^{(i)}_j$
are smaller than those representing $T$.  In section \ref{ghost} we described the simple structure
of the matrices $L$, which represent $T$ in the standard basis.
We have not attempted to construct the more complicated
matrices which represent $T^{(i)}_j$.

We use a second approach. We compute eigenvalues and eigenvectors of the full operator
$T$, and use Corollary \ref{SchurCor} to classify the eigenvalues according to
the corresponding irreducible representation.  An eigenvector $\psi$ of
$T$ will satisfy $P^{(i)} (\psi) = \psi$ for exactly one $i$, barring
accidental degeneracy.  By computing all the projections, we can determine which
irreducible representation is associated with the eigenvector.
If the irreducible representation has dimension greater than one, then we can choose an orthonormal
set of eigenvectors such that each eigenvector lies in a unique $V^{(i)}_j$ using the projections $P^{(i)}_j$.

For our group $G = \D_6$, we do not have to use the full projection operators
$P^{(i)}$ and $P^{(i)}_j$.
We only need the four projection operators
$$
P_{p_x p_y} = \frac{\alpha_{1} + p_x \al_\sigma + p_y \al_\tau + p_x p_y \al_{\rho^3}}{4},
$$
where $p_x, p_y \in \{+, - \}$.
For example $P_{+-} (u) = (u + \sigma \cdot u - \tau \cdot u - \rho^3 \cdot u)/4$.
Note that $P_{p_x p_y}(U) = V_{p_x p_y 1} \oplus V_{p_x p_y 2}$.  If $\psi$ is
an eigenvector with multiplicity 1 and $P_{p_x p_y} \psi \neq 0$, then $P_{p_x p_y} \psi = \psi$
and $\psi \in V_{p_x p_y 1}$.  It is noteworthy that we do not have
to check for rotational symmetry.  Similarly, if $\psi$ is an eigenvector with multiplicity 2
and $P_{p_x p_y} \psi \neq 0$, then $P_{p_x p_y} \psi \in V_{p_x p_y 2}$.

Here is our algorithm to compute the symmetry of the eigenfunctions of $T$, and to
find eigenvectors in the $T$-invariant spaces $V^{(i)}_j$.
First of all, we know from the computed eigenvalues which are simple and which are double.
If an eigenvalue $\lam$ is simple, then the corresponding eigenvector
$\psi$ satisfies $P_{p_x p_y} \psi = \psi$
for exactly one choice of $p_x $ and $p_y$, and $\psi \in V_{p_x p_y 1}$.
If an eigenvalue $\lam$ has multiplicity 2, then each of the 2 orthogonal eigenfunctions
returned by ARPACK is replaced with the largest of the four projections $P_{p_x p_y} \psi$.
Then, the projected eigenfunctions are normalized.
If necessary, the eigenfunctions are reordered so that $\psi \in V_{++2}$ comes
before $\psi \in V_{--2}$, and $\psi \in V_{+-2}$ comes before $\psi \in V_{-+2}$.
This algorithm assumes that there is no accidental degeneracy.  We produced a list of the
eigenvalues and symmetry types up to the $300^{\rm th}$ eigenvalue, confirming that there is
no accidental degeneracy.

\begin{figure}[hbt]
\begin{center}
\begin{tabular}{|c|c|c|c|}
\hline
 & & & \\
    \scalebox{.25}{\includegraphics{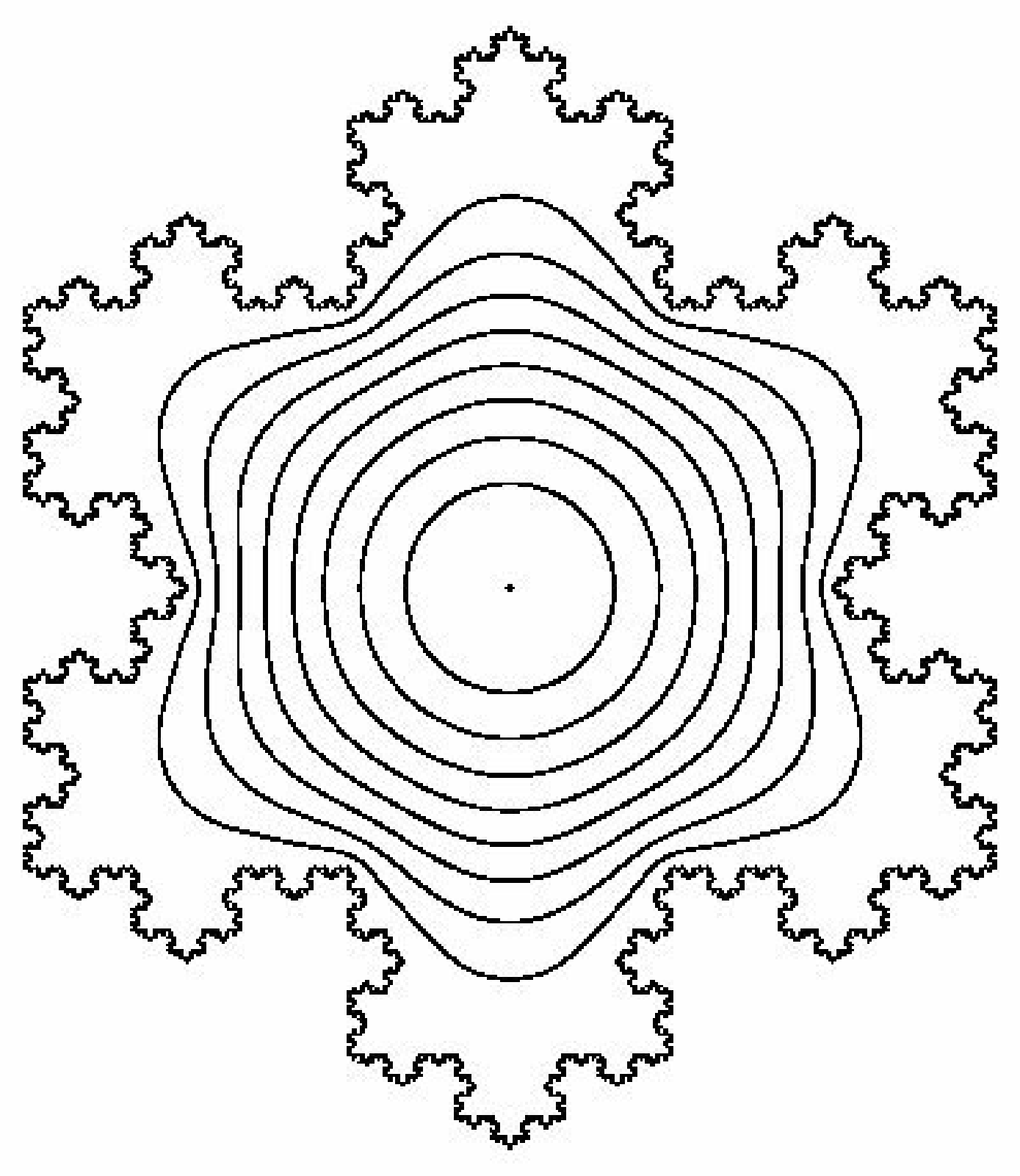}}

&
    \scalebox{.25}{\includegraphics{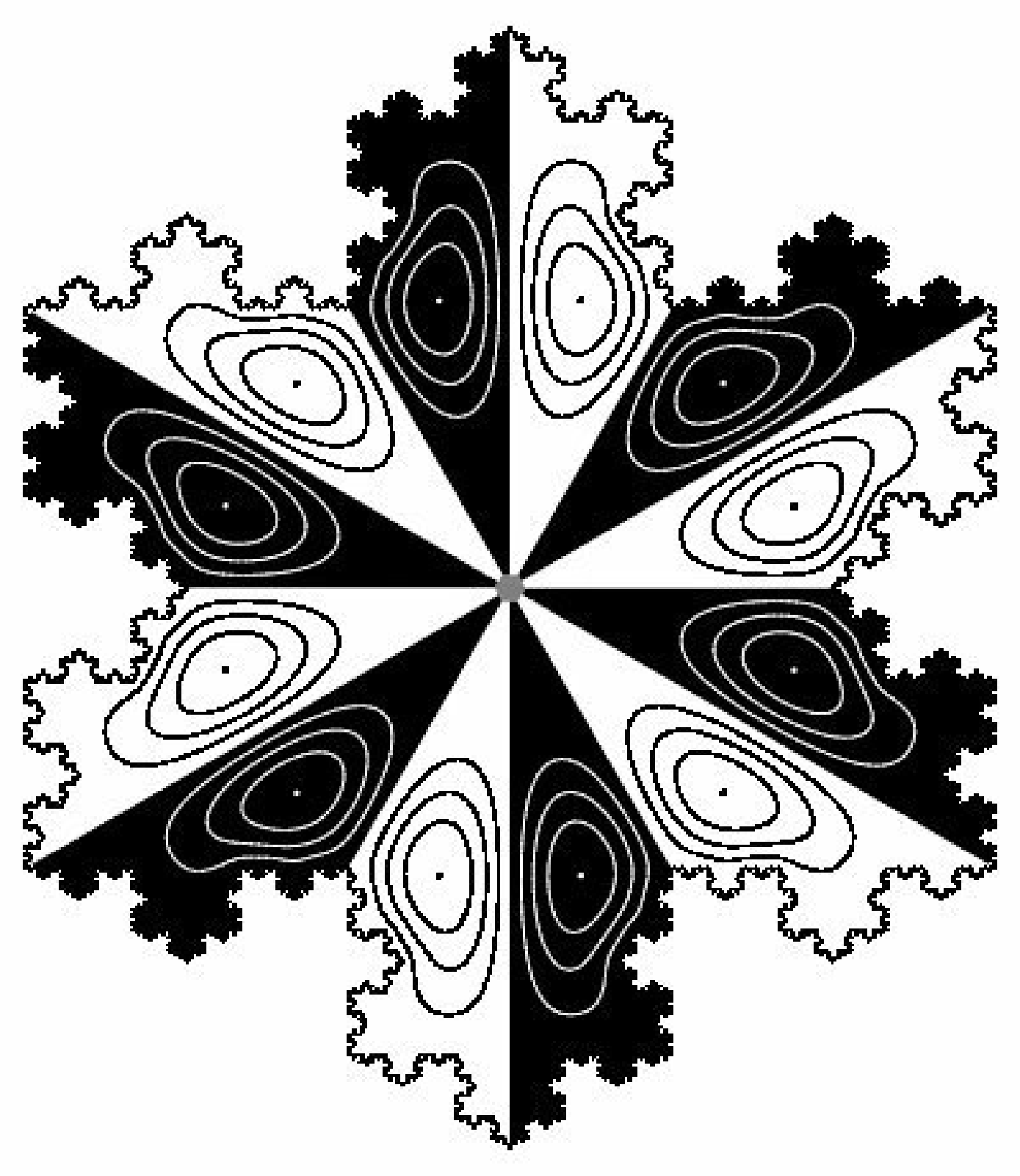}}
&
    \scalebox{.25}{\includegraphics{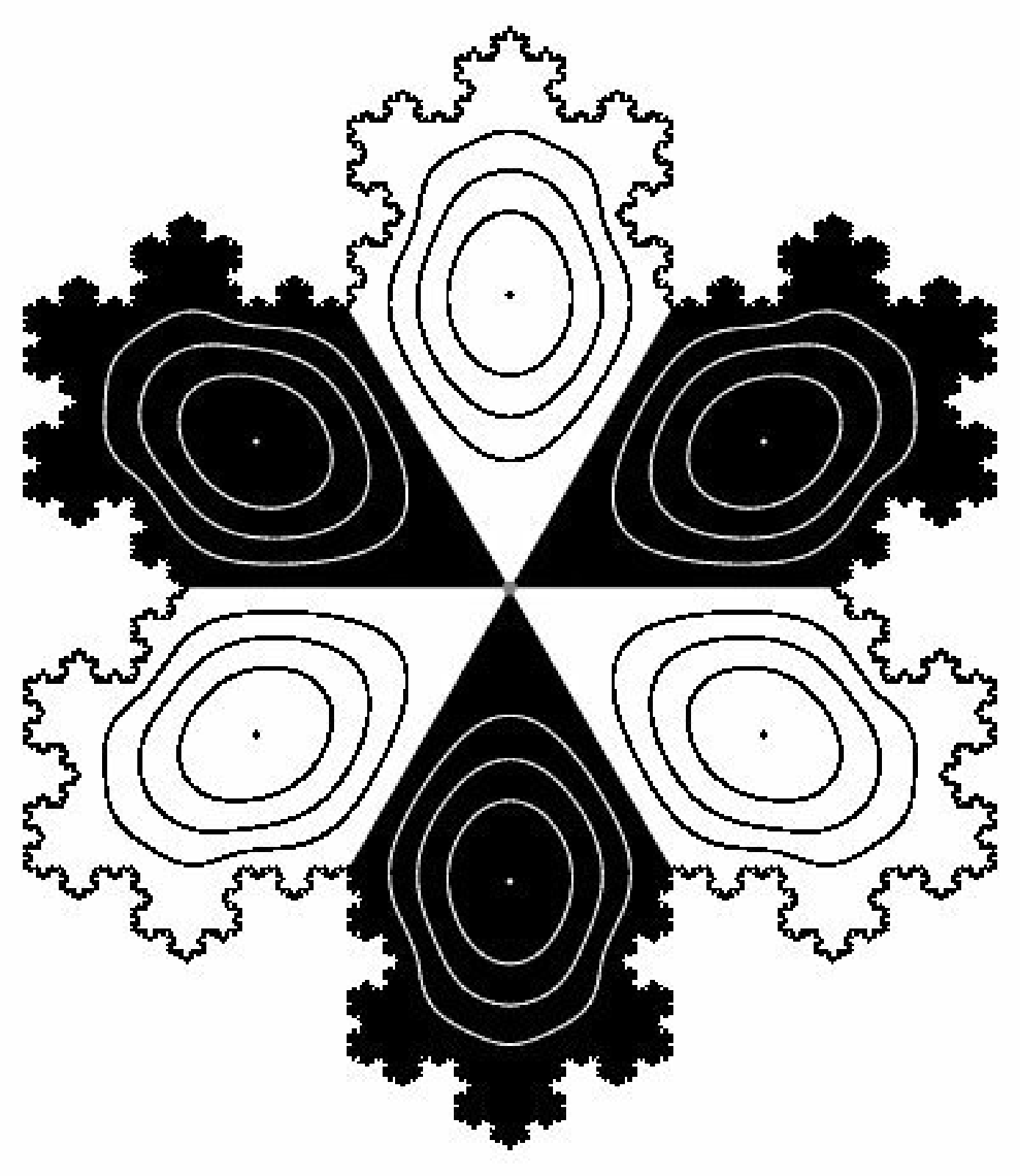}}
&
    \scalebox{.25}{\includegraphics{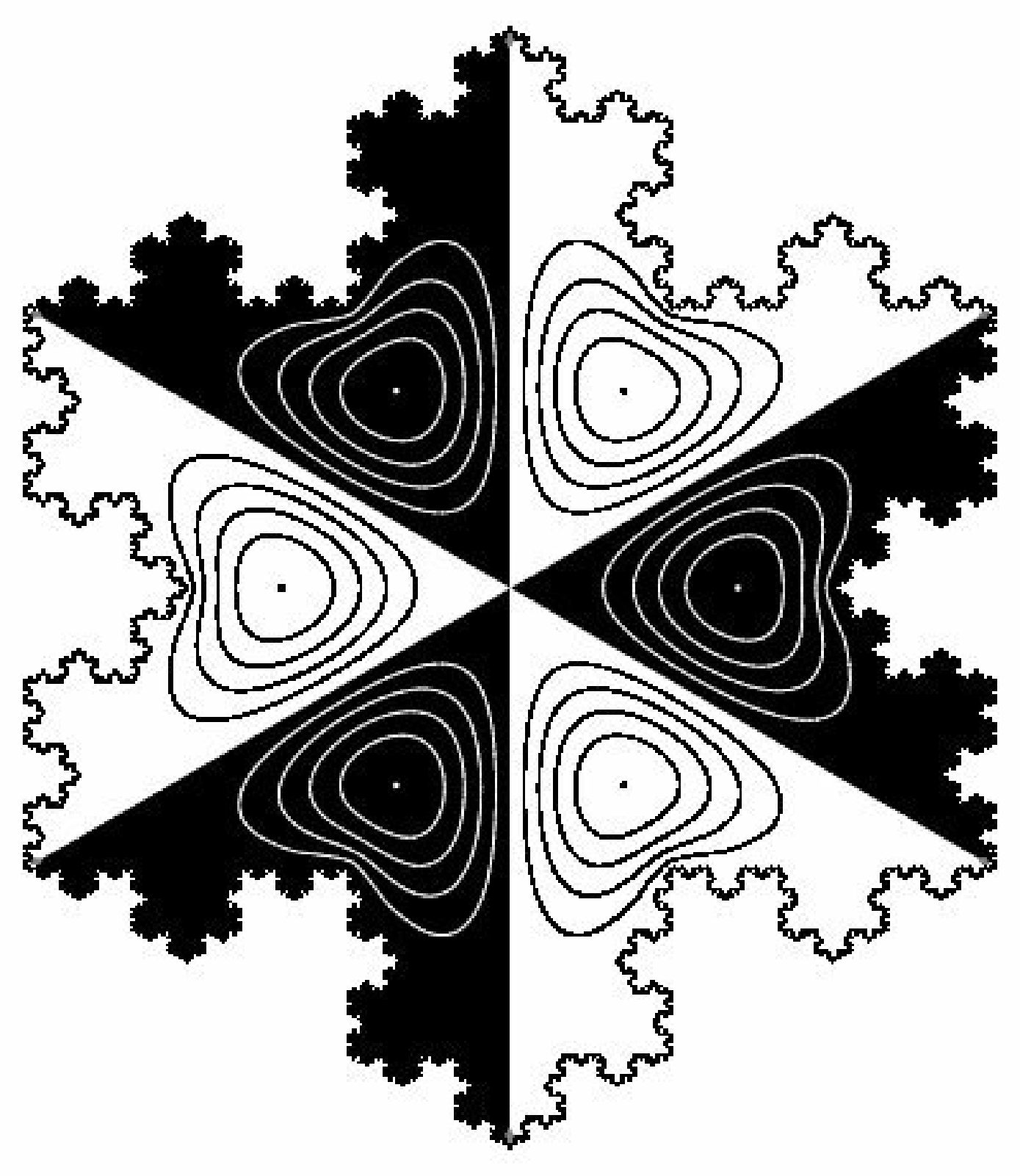}}
    \\

$\psi_1\in V_{++1} := V^{(1)}$ & $\psi_{24}\in V_{--1} := V^{(2)}$ &
$\psi_7\in V_{+-1} := V^{(3)}$ & $\psi_{10}\in V_{-+1} := V^{(4)}$ \\

\hline  & & & \\
    \scalebox{.25}{\includegraphics{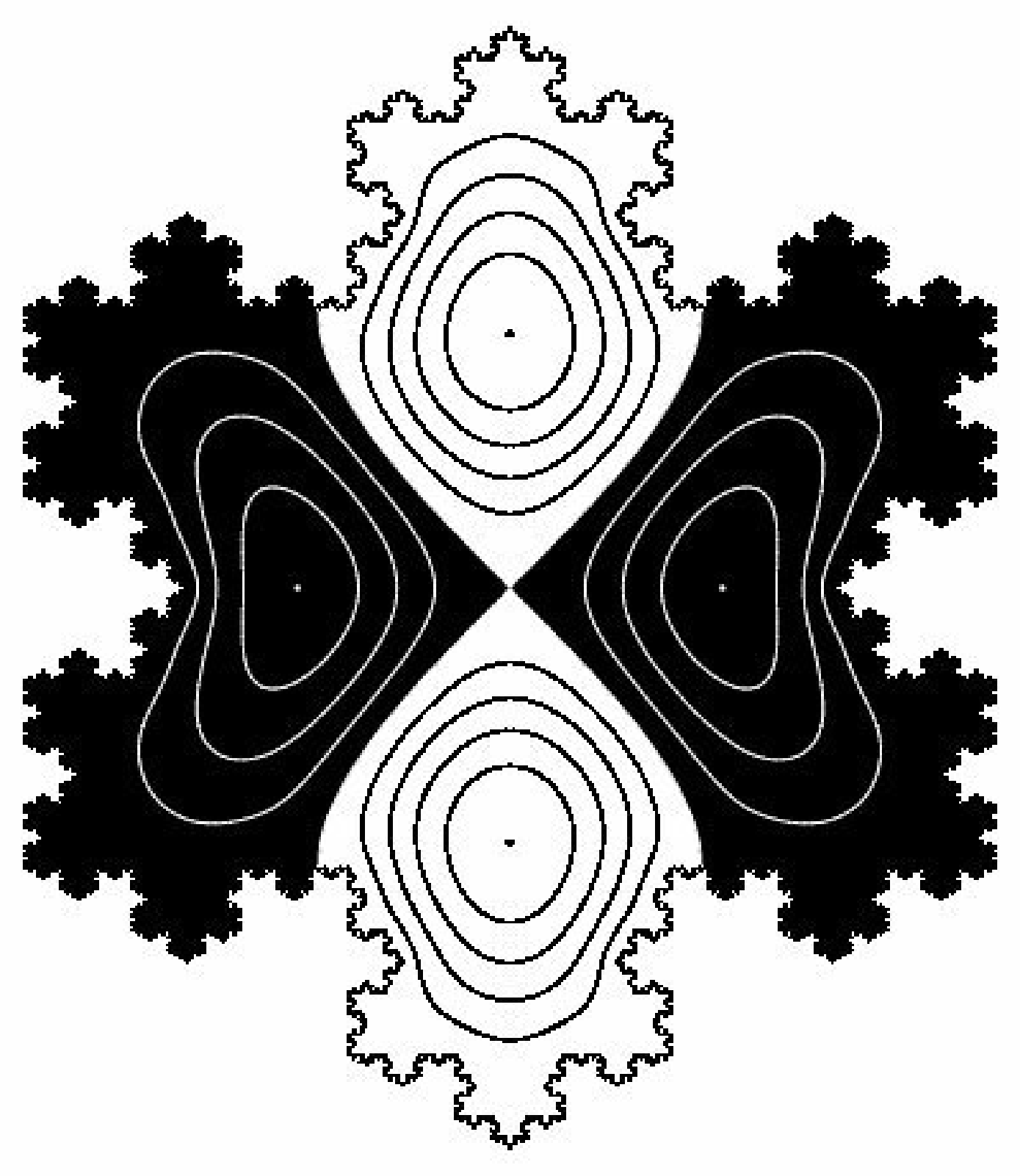}}
&
    \scalebox{.25}{\includegraphics{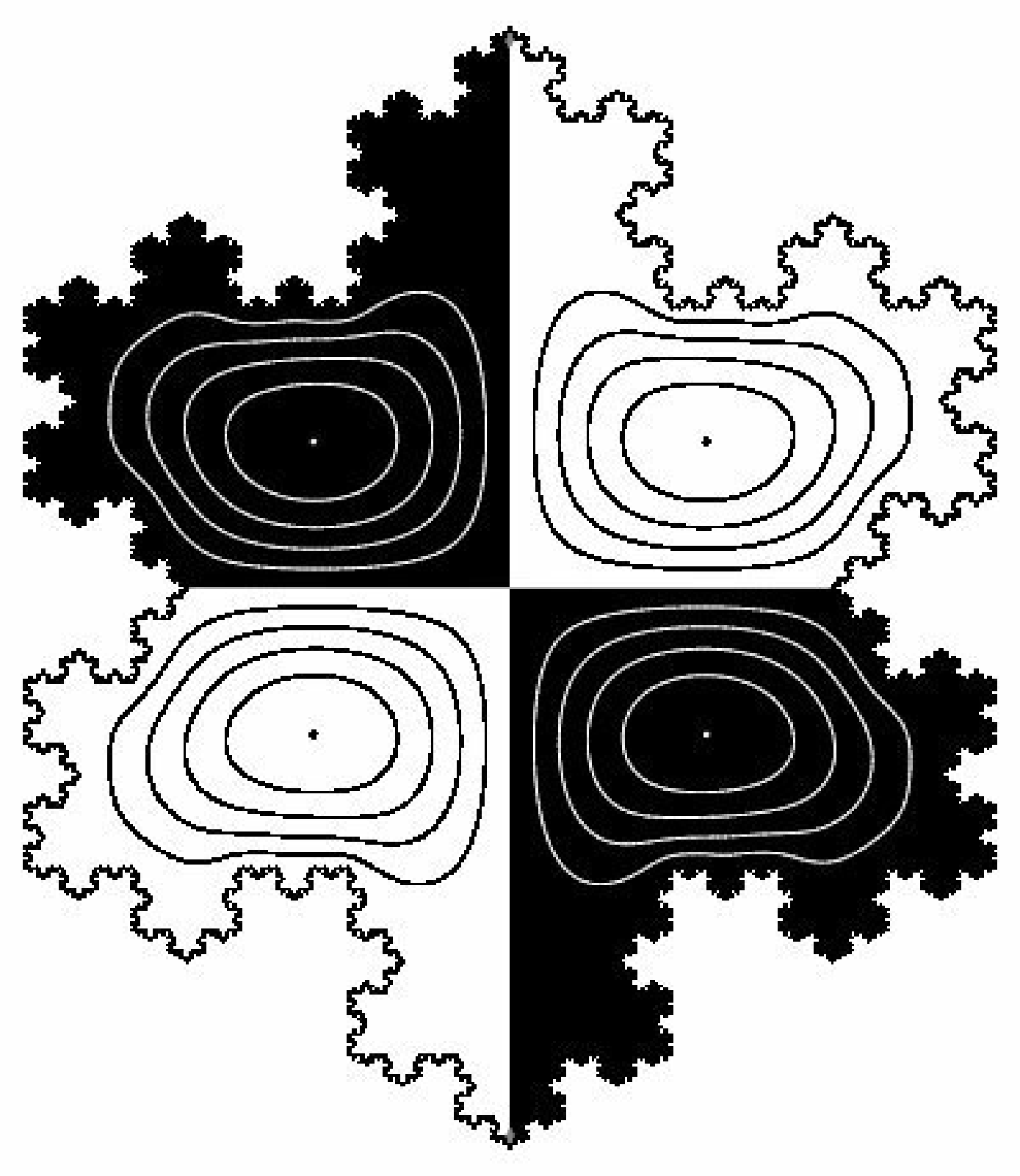}}
&
    \scalebox{.25}{\includegraphics{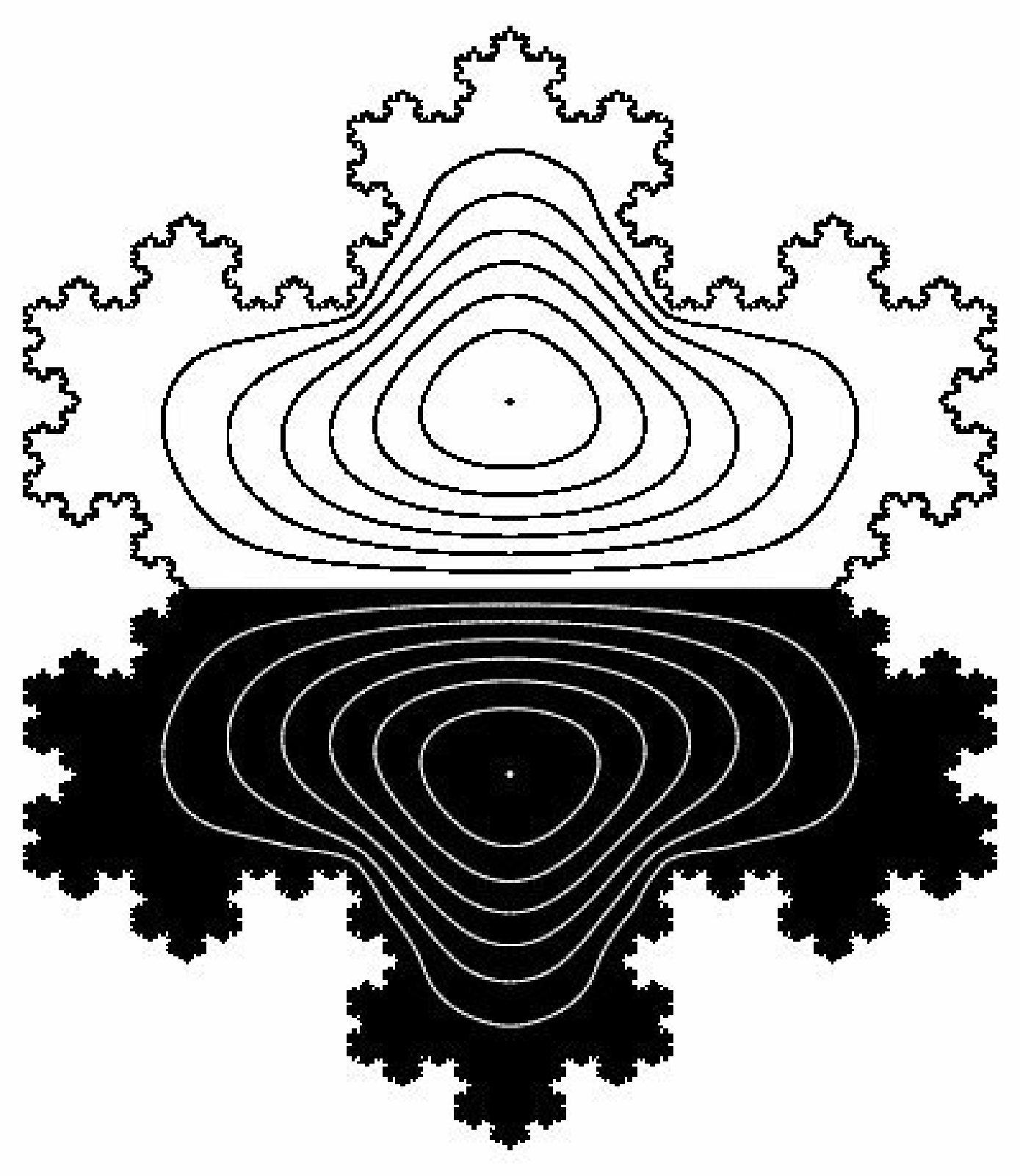}}
&
    \scalebox{.25}{\includegraphics{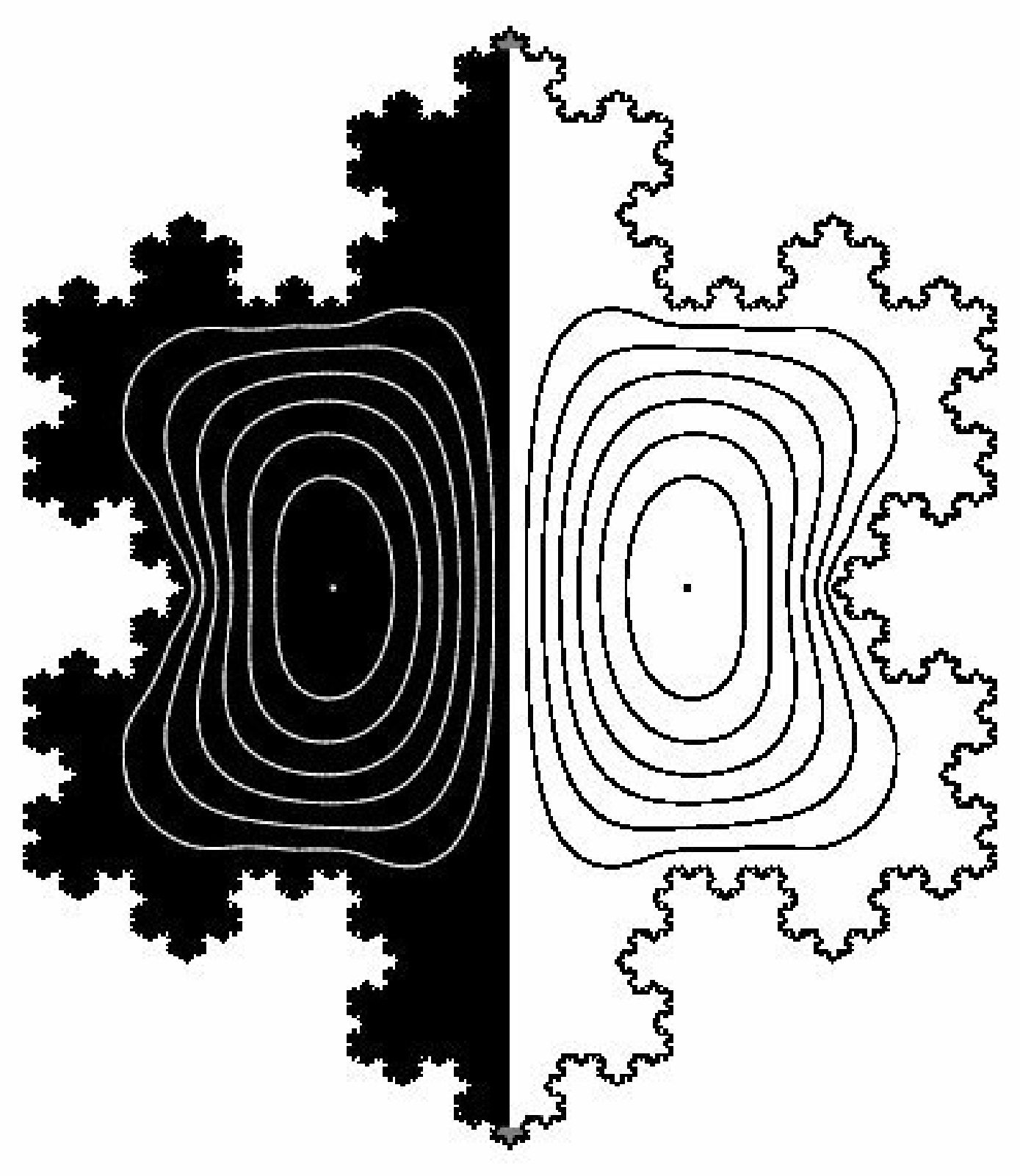}}
    \\
$\psi_4\in V_{++2} := V^{(5)}_1$ & $\psi_{5}\in V_{--2} := V^{(5)}_2$ &
$\psi_2\in V_{+-2} := V^{(6)}_1$ & $\psi_{3}\in V_{-+2} := V^{(6)}_2$
\\ \hline
\end{tabular}
\vspace*{.2in}
\caption{The first occurrences of the 8 symmetry types of
eigenfunctions of the Laplacian with zero Dirichlet boundary conditions.
Each of the eigenfunctions in the first row has a simple eigenvalue.
Those in the second row come in pairs.
For example,
$\lam_4 = \lam_5$, since $\psi_4$ and $\psi_5$ are each in the invariant subspace
$V^{(5)}$ corresponding to the 2-dimensional irreducible representation $\Gamma^{(5)}$.
Similarly,
$\lam_2 = \lam_3$, since $\psi_2$ and $\psi_3$ are each in $V^{(6)}$.
The eigenfunctions shown respect the canonical decompositions $V^{(5)} = V^{(5)}_1 \oplus V^{(5)}_2$
and $V^{(6)} = V^{(6)}_1 \oplus V^{(6)}_2$.
}
\label{eight_sym}
\end{center}
\end{figure}

\begin{figure}[hbt]
\begin{center}
\begin{tabular}{|c|c|c|c|}
\hline
 & & & \\
    \scalebox{.25}{\includegraphics{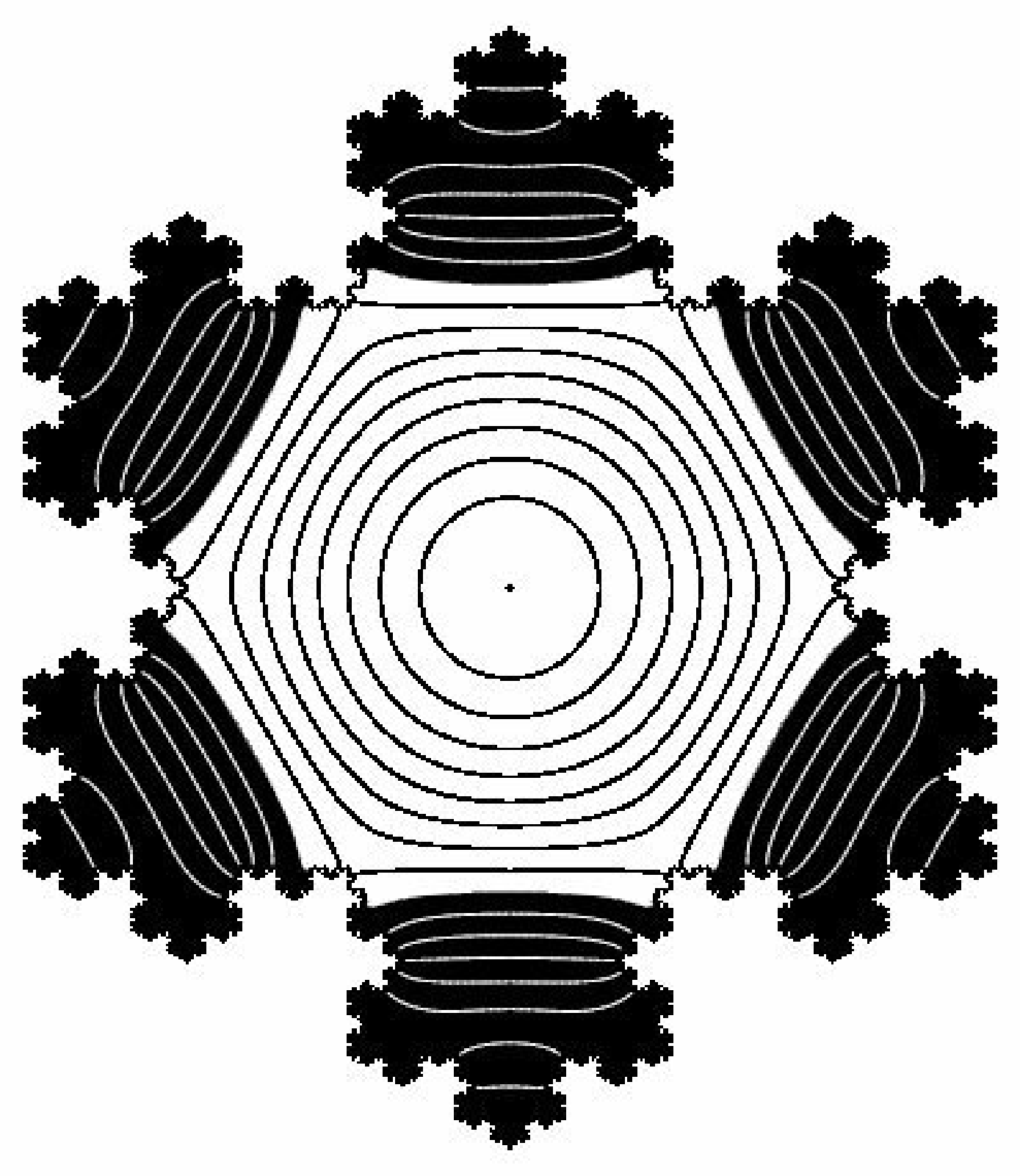}}
&
    \scalebox{.25}{\includegraphics{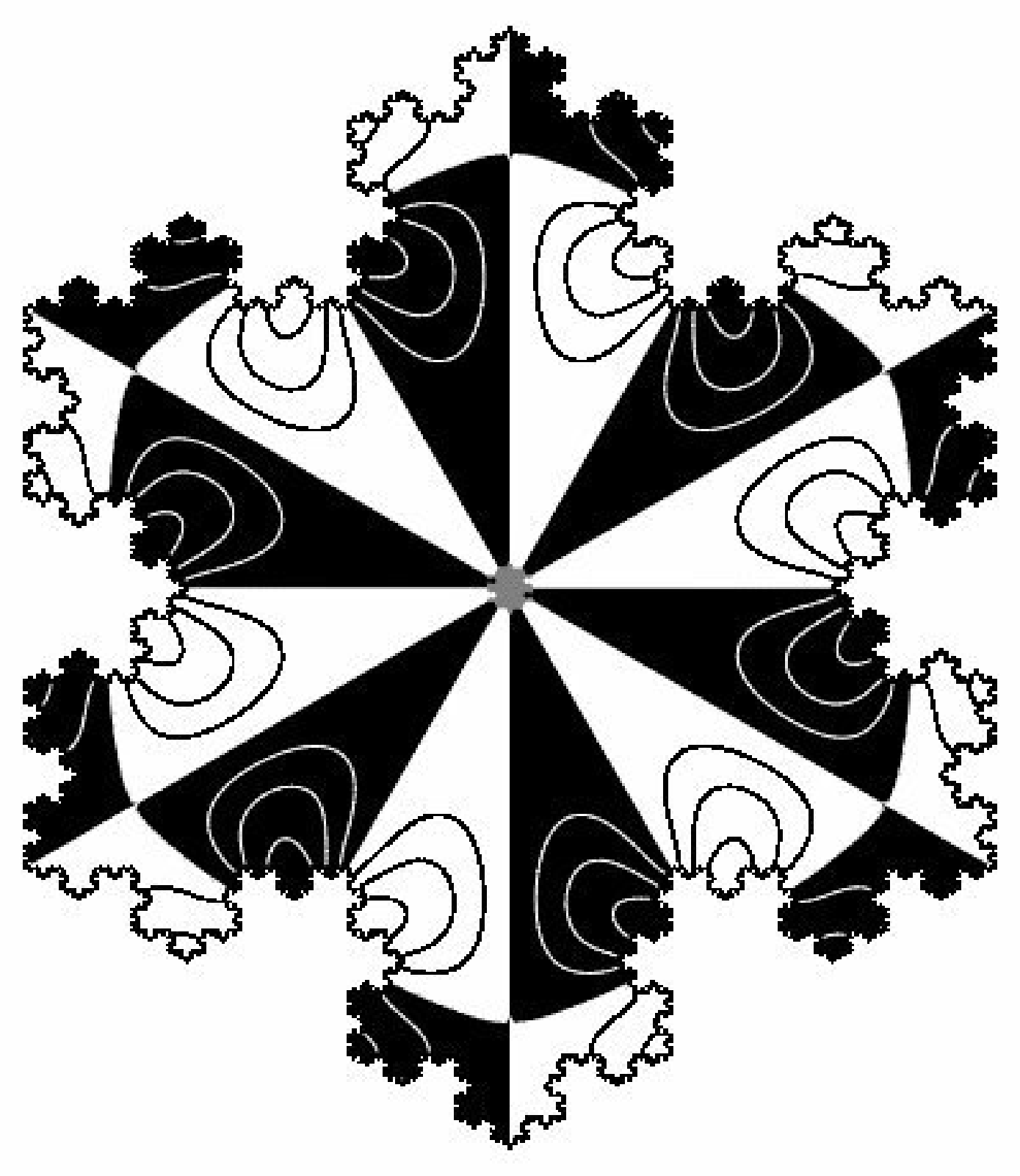}}
&
    \scalebox{.25}{\includegraphics{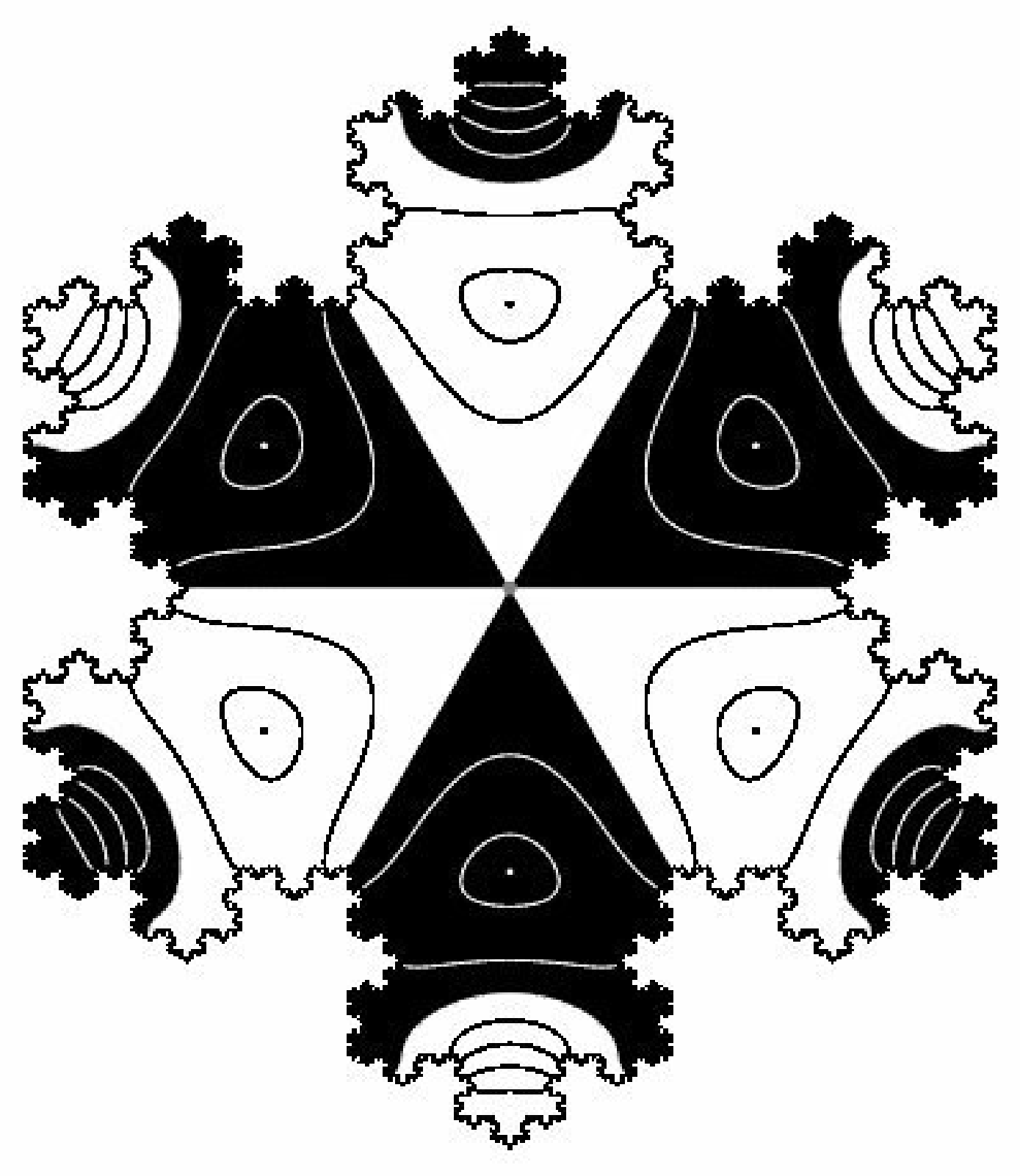}}
&
    \scalebox{.25}{\includegraphics{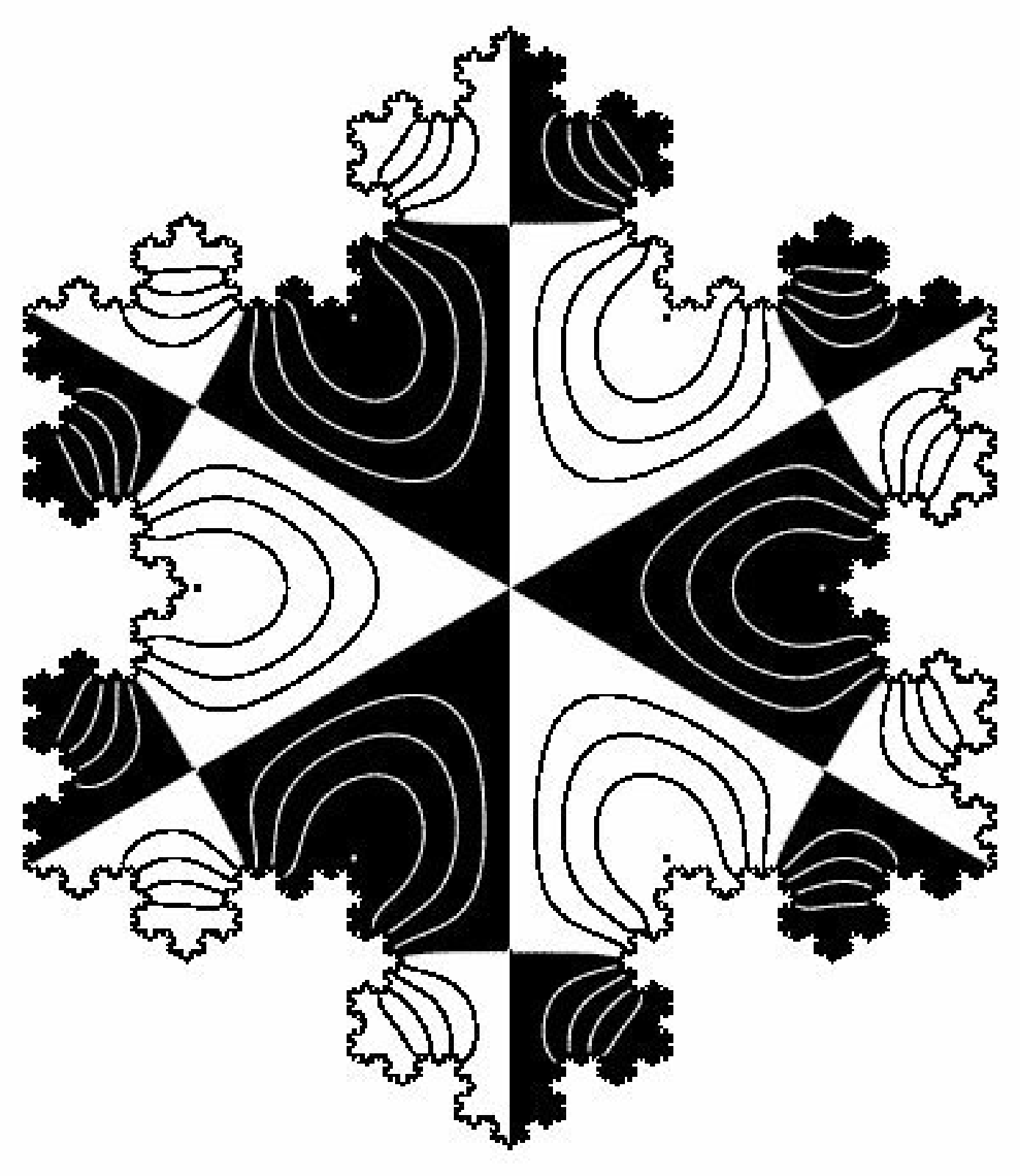}}
    \\

$\psi_7\in V_{++1} := V^{(1)}$ & $\psi_{36}\in V_{--1} := V^{(2)}$ &
$\psi_{19}\in V_{+-1} := V^{(3)}$ & $\psi_{20}\in V_{-+1} := V^{(4)}$ \\

\hline  & & & \\
    \scalebox{.25}{\includegraphics{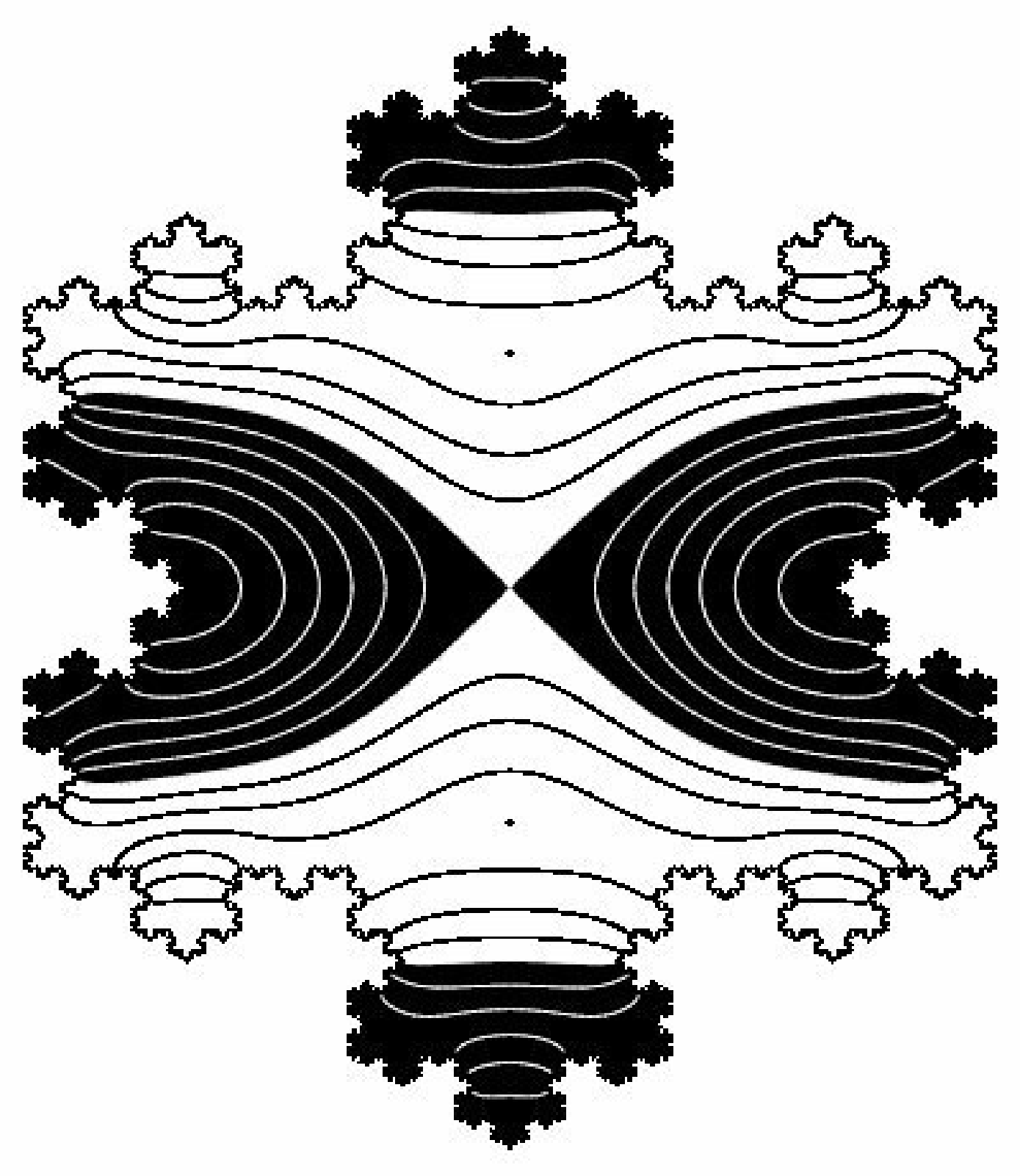}}
&
    \scalebox{.25}{\includegraphics{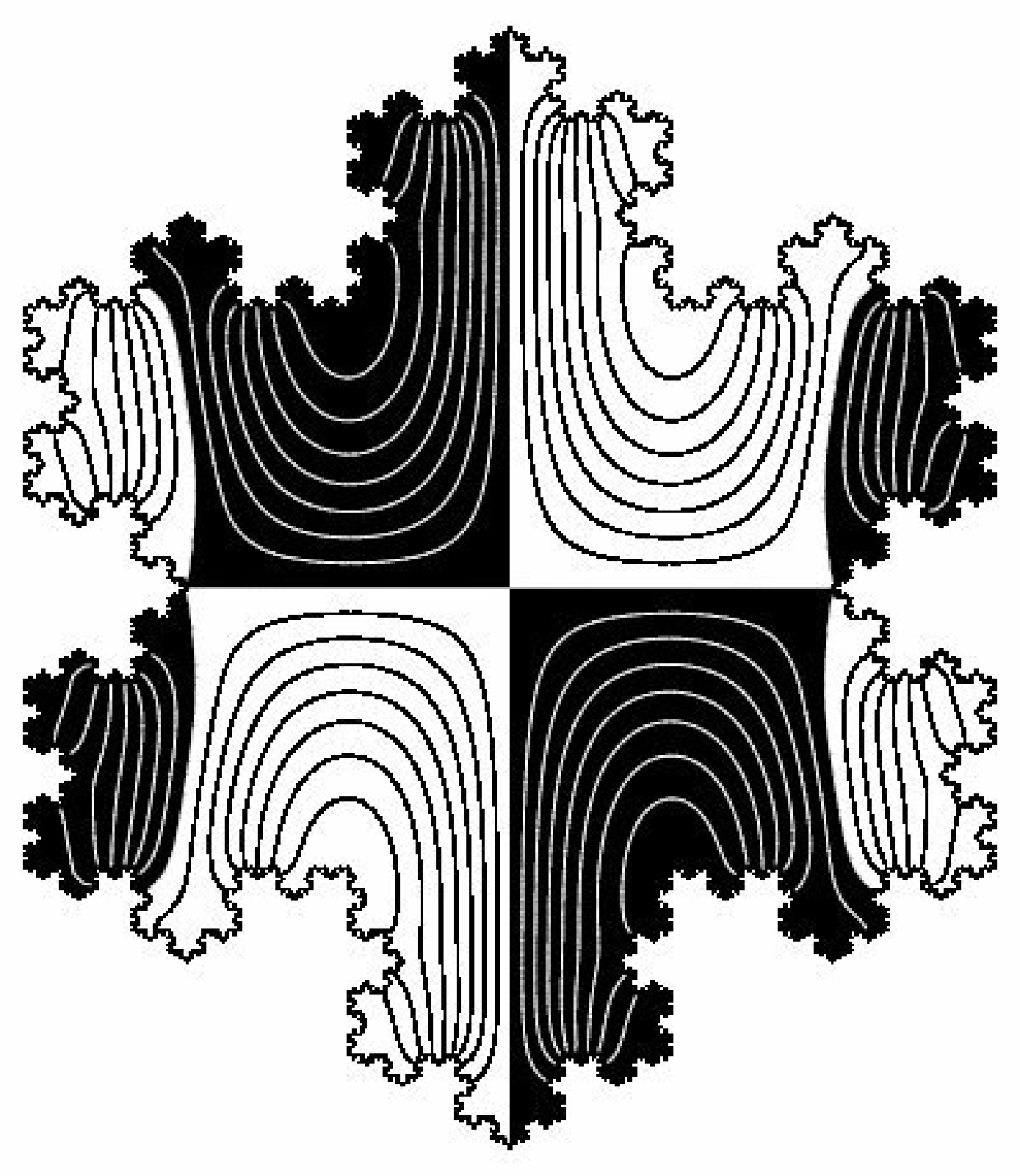}}
&
    \scalebox{.25}{\includegraphics{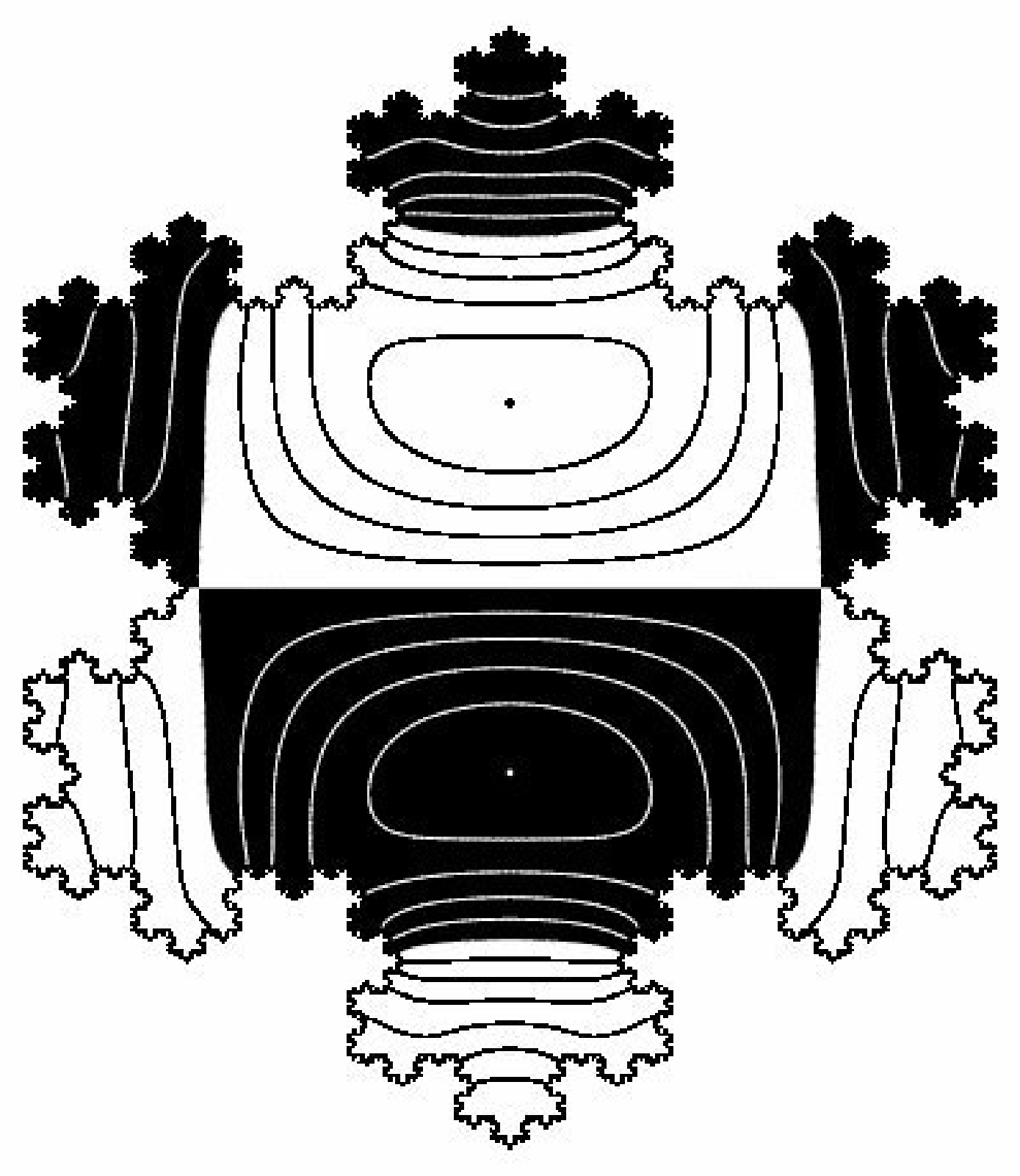}}
&
    \scalebox{.25}{\includegraphics{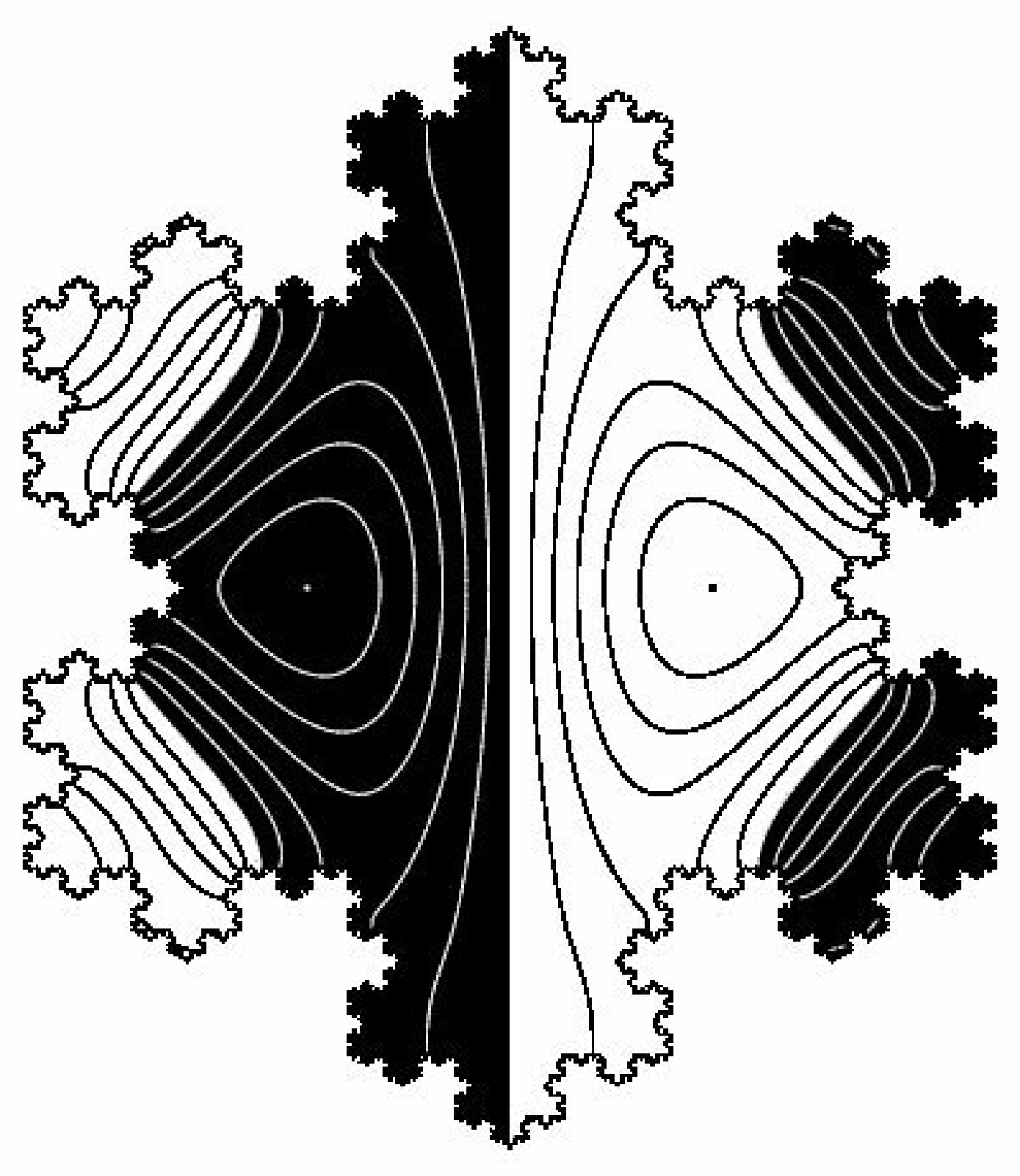}}
    \\
$\psi_{10}\in V_{++2} := V^{(5)}_1$ & $\psi_{11}\in V_{--2} := V^{(5)}_2$ &
$\psi_{8}\in V_{+-2} := V^{(6)}_1$ & $\psi_{9}\in V_{-+2} := V^{(6)}_2$
\\ \hline
\end{tabular}
\vspace*{.2in}
\caption{The {\em second} occurrences of the 8 symmetry types of
eigenfunctions of the Laplacian, with Neumann boundary conditions.
Compare with Figure \ref{eight_sym}.
}
\label{eight_sym_Neumann}
\end{center}
\end{figure}

We now explain why replacing each eigenfunction with the largest projection gives two linearly independent
eigenfunctions.
Suppose $\psi$ and $\phi$ are two eigenfunctions with the same eigenvalue.
Then either $\psi, \phi \in V^{(5)} = V_{++2} \oplus V_{--2}$ or
$\psi, \phi \in V^{(6)} = V_{+-2} \oplus V_{-+2}$.
Assume that we are in the first case.  The eigenfunctions are orthonormal, so
that $||\psi||^2 = ||\phi||^2 = 1$ and the inner product $(\psi, \phi) = 0$.  We want to
replace $\psi$ and $\phi$ with normalized eigenfunctions
$\psi_{++} \in V_{++ 2}$ and $\psi_{--} \in V_{-- 2}$.
We can write
$\psi = a_\psi \psi_{++} + b_\psi \psi_{--}$ and $\phi = a_\phi \psi_{++} + b_\phi \psi_{--}$,
where $||\psi||^2 = a_\psi^2 + b_\psi^2 = 1$, $||\phi||^2 = a_\phi^2 + b_\phi^2 = 1$,
and $(\psi, \phi) = a_\psi a_\phi + b_\psi b_\phi = 0$.  In other words,
$\langle a_\phi, b_\phi \rangle$ and
$\langle a_\psi, b_\psi \rangle$
are two orthogonal unit vectors in the $a$-$b$ plane.  Therefore, $a_\psi^2 \geq
b_\psi^2$ if and only if $b_\phi^2 \leq a_\phi^2$.  Neglecting the numerical
coincidence
of $a_\psi^2 = b_\psi^2$, this implies that replacing $\psi$ and $\phi$
by their largest projection will include eigenfunctions of both symmetry types.
A similar argument holds for pairs of eigenfunctions in $V_{+-2} \oplus V_{-+2}$.

Figures \ref{eight_sym}  and \ref{eight_sym_Neumann} shows eigenfunctions with each of the 8 symmetry types
for the Dirichlet and Neumann eigenvalue problems, respectively.
More eigenfunction contour plots can be found at \texttt{http://jan.ucc.nau.edu/$\sim$ns46/snow/}.

\end{section}
\begin{section}{Conclusions.}
Table \ref{summary} summarizes our best estimates of the eigenvalues, and the
symmetry types of the eigenfunctions, for the first few eigenvalues of the
linear problem (\ref{lpde}).
We have no error bounds
but we can get an indication of the accuracy by comparing
our results with those in Banjai's thesis \cite{banjai}.
Banjai computed the first 20 eigenvalues of the Dirichlet problem,
and our eigenvalues agree with his results
to four significant figures. 
This suggests that our results are accurate to four significant figures.
Banjai's results for the first ten Dirichlet eigenvalues are more precise.
\begin{table}[htbp]
\label{summary}
\begin{center}
\begin{tabular}{|c|cc|cc|}
\hline
$k$ & $\lam^{\rm R}_k$ (D) & Symmetry & $\lam^{\rm R}_k$ (N) & Symmetry  \\
\hline
1   &   39.349   &   $   ++1   $   &   0   &   $   ++1   $   \\
2   &   97.438   &   $   +-2   $   &   11.842   &   $   +-2   $   \\
3   &   97.438   &   $   -+2   $   &   11.842   &   $   -+2   $   \\
4   &   165.41   &   $   ++2   $   &   23.047   &   $   ++2   $   \\
5   &   165.41   &   $   --2   $   &   23.047   &   $   --2   $   \\
6   &   190.37   &   $   ++1   $   &   27.426   &   $   +-1   $   \\
7   &   208.62   &   $   +-1   $   &   52.210   &   $   ++1   $   \\
8   &   272.41   &   $   +-2   $   &   85.552   &   $   +-2   $   \\
9   &   272.41   &   $   -+2   $   &   85.552   &   $   -+2   $   \\
10   &   312.36   &   $   -+1   $   &   112.02   &   $   ++2   $   \\
11   &   314.45   &   $   ++2   $   &   112.02   &   $   --2   $   \\
12   &   314.45   &   $   --2   $   &   118.34   &   $   -+1   $   \\
13   &   359.53   &   $   ++1   $   &   139.38   &   $   +-2   $   \\
14   &   425.40   &   $   +-1   $   &   139.38   &   $   -+2   $   \\
15   &   443.54   &   $   +-2   $   &   139.54   &   $   --1   $   \\
16   &   443.54   &   $   -+2   $   &   147.64   &   $   ++1   $   \\
17   &   458.66   &   $   ++2   $   &   151.00   &   $   ++2   $   \\
18   &   458.66   &   $   --2   $   &   151.00   &   $   --2   $   \\
19   &   560.41   &   $   +-2   $   &   183.64   &   $   +-1   $   \\
20   &   560.41   &   $   -+2   $   &   197.50   &   $   -+1   $   \\
21   &   566.79   &   $   ++2   $   &   207.10   &   $   +-2   $   \\
22   &   566.79   &   $   --2   $   &   207.10   &   $   -+2   $   \\
23   &   595.18   &   $   ++1   $   &   216.81   &   $   ++2   $   \\
24   &   617.70  &   $   --1   $   &   216.81   &   $   --2   $   \\
\hline
\end{tabular}
\vspace*{.2in}
\caption{
The first 24 eigenvalues of the negative Laplacian, and the symmetry of the
corresponding eigenfunctions, with Dirichlet or Neumann boundary conditions.
The Richardson extrapolations for the eigenvalues are given.
A comparison with Banjai's thesis suggests that our results are
accurate to four
significant figures.
The symmetry
indicates the space containing the eigenfunction.  For example,
the first eigenfunction is in $V_{++1}$ for both boundary conditions.  We follow the convention
that $V_{++2}$ is listed before $V_{--2}$, and $V_{+-2}$ is listed before
$V_{-+2}$.}
\end{center}
\end{table}

We are ultimately interested in the connections between the linear problem (\ref{lpde})
and superlinear elliptic boundary value problems of the form
\begin{align}
\label{pde}
        \Delta u + f(u) &= 0  \textrm{ in } \Omega \nonumber \\
         u &= 0  \textrm{ on } {\partial \Omega}.
\end{align}
The so-called Gradient-Newton-Galerkin-Algorithm (GNGA, see [NS]) seeks
approximate solutions $u=\sum_{j=1}^{M} a_j \psi^{\ell}_j$ to (\ref{pde})
by applying Newton's method to the eigenfunction expansion coefficients of the gradient
$\nabla J(u)$ of a nonlinear functional $J$ whose critical points are the desired solutions.
In [NSS2], we will enumerate the 23 isotropy subgroups of $\D_6 \times \Z_2$,
along with the associated fixed point subspaces.
These are the possible symmetry types of solutions to (\ref{pde}) when $f$ is odd.
As a result, we will be able to follow the bifurcation branches of (\ref{pde})
by varying the parameter $\lambda$ in the nonlinearity defined by $f(u)=\lambda u + u^3$.
Our symmetry information is crucial to the branch continuation decision making process.
In [NSS2] we will find solutions of all 23 symmetry types.

Our catalog of symmetry information is helpful in another way.
Given that our basis is chosen in a symmetric fashion and
by knowing which invariant subspace a given solution branch's elements lie,
it is often possible to know that many of the eigenfunction expansion coefficients are zero.
The GNGA requires numerical integration using values at the $N=N_{\rm NSS}(\ell)$ grid points
for each coefficient of the gradient $\nabla J(u)$,
as well as for each entry of the Hessian matrix representing $D_2 J(u)$.
If we use a basis of eigenfunctions spanning a subspace of dimension $M$,
this amounts to $M^2+M$ integrations.
Using symmetry, we are able to avoid many of these integrations, resulting in a substantial speedup.
This becomes very important as we seek high-energy solutions with complex nodal structure,
where large values of $M$ and $N$ are required.
Each solution is represented by a point on a bifurcation curve, whereby many solutions are
sought to complete the branch.
Many of the branches of high-energy solutions have a rich bifurcation structure, resulting
in multiple secondary and tertiary branches.
There is a substantial time savings obtained by
cutting down on the number of required integrations for each Newton step
for each solution, on each branch.

In conclusion, this paper presents an efficient, accurate, and easy to implement method for obtaining a basis of
eigenfunctions to a subspace which is sufficiently large for performing the eigenfunction expansions required
by the GNGA method in solving nonlinear problems of the form (\ref{pde}).
The major innovations of this work are in the new way we enforce boundary conditions and in the
way we are able to choose symmetric representatives of eigenfunctions corresponding to multiple eigenvalues.
Our new techniques for generating contour plots are also noteworthy.
By further decomposing function space according to symmetry,
our nonlinear experiments are much more successful (see \cite{nss2}).
We are currently porting the code to a small cluster using the parallel implementation PARPACK.
In the future, we expect to use the techniques of this paper to generate larger numbers of
more accurate eigenvalue/eigenfunction approximations, and in so doing, be able to investigate more
complicated phenomena and/or problems with regions $\Omega$ in ${\mathbb R}^n$, for dimensions $n=3$ and higher.
\end{section}


\end{document}